\documentclass[reqno,11pt]{amsart}

\usepackage{amsmath,amssymb,amsthm,mathrsfs}
\usepackage{epsfig,color}
\usepackage[latin1]{inputenc}

\usepackage[numbers]{natbib}

\numberwithin{equation}{section}

\def\H{\mathcal H}

\def\Le{\mathcal L}
\def\R{\mathbb R}
\def\N{\mathbb N}

\newcommand{\dist}{\mathop{\mathrm{dist}}}
\def\e{\varepsilon}
\def\s{\sigma}
\def\S{\Sigma}
\def\vphi{\varphi}
\def\Div{{\rm div}\,}
\def\om{\omega}
\def\l{\lambda}
\def\g{\gamma}
\def\k{\kappa}
\def\Om{\Omega}
\def\de{\delta}
\def\Id{{\rm Id}}

\def\spt{{\rm spt}}

\def\pa{\partial}
\def\trace{{\rm tr}}

\def\00{{\bf 0}}

\def\F{\mathcal{F}}

\renewcommand{\a}{\alpha}
\renewcommand{\b}{\beta}

\newcommand{\hd}{\mathrm{hd}}
\newcommand{\D}{\Delta}
\renewcommand{\l}{\lambda}

\renewcommand{\om}{\omega}

\newcommand{\tr}{\mbox{tr }}

\renewcommand{\Div}{{\rm div \,}}
\newcommand{\ov}{\overline}

\newcommand{\diam}{\mathrm{diam}}
\newcommand{\cc}{\subset\subset}

\def\weak{\stackrel{*}{\rightharpoonup}}

\def\p{\mathbf{p}}

\def\C{\mathbf{C}}
\def\D{\mathbf{D}}

\newtheorem*{theorem*}{Theorem}
\newtheorem{theorem}{Theorem}[section]
\newtheorem{lemma}[theorem]{Lemma}
\newtheorem{proposition}[theorem]{Proposition}

\newtheorem{remark}[theorem]{Remark}

\title[Hypersurfaces with almost constant mean curvature]{On the shape of compact hypersurfaces with almost constant mean curvature}

\author{G. Ciraolo}
\address{Dipartimento di Matematica e Informatica,
Università di Palermo, Via Archirafi 34, 90123 Palermo, Italy}
\email{giulio.ciraolo@unipa.it}

\author{F. Maggi}
\address{Department of Mathematics, University of Texas at Austin, Austin, TX, USA}
\email{maggi@math.utexas.edu}

\begin{document}

\begin{abstract} The distance of an almost constant mean curvature boundary from a finite family of disjoint tangent balls with equal radii  is quantitatively controlled in terms of the oscillation of the scalar mean curvature. This result allows one to quantitatively describe the geometry of volume-constrained stationary sets in capillarity problems.
\end{abstract}

\maketitle

\section{Introduction}
We investigate the geometry of compact boundaries with almost constant mean curvature in $\R^{n+1}$, $n\ge 2$. Beyond its intrinsic geometric interest, this problem is motivated by the description of equilibrium shapes (volume-constrained stationary sets) of the classical Gauss free energy used in capillarity theory \cite{Finn}, and consisting of a dominating surface tension energy plus a potential energy term. Our analysis leads to new stability estimates describing in a quantitative way the distance of these shapes from compounds of tangent balls of equal radii.

\subsection{Main result} Given a connected bounded open set $\Om\subset\R^{n+1}$ ($n\ge 2$) with $C^2$-boundary, we denote by $H$ the scalar mean curvature of $\pa\Om$ with respect to the outer unit normal $\nu_\Om$ to $\Om$ (normalized so that $H=n$ if $\Om=B=\{x\in\R^{n+1}:|x|<1\}$), and we introduce the {\it Alexandrov's deficit} of $\Om$,
\begin{equation}
  \label{alexandrov deficit}
  \de(\Om)=\frac{\|H-H_0\|_{C^0(\pa\Om)}}{H_0}\,,\qquad\mbox{where}\qquad H_0=\frac{n\, P(\Om)}{(n+1)|\Om|}\,.
\end{equation}
This is a scale invariant quantity (i.e. $\de(\Om)=\de(\l\Om)$ for every $\l>0$) with the property that $\de(\Om)=0$ if and only if $\Om$ is a ball (Alexandrov's theorem). The motivation for the particular value of $H_0$ considered in the definition of $\de(\Om)$ is that if $H$ is constant on $\pa\Om$, then by the divergence theorem it must be $H=H_0$ (see \eqref{computation} below). Here and in the following, $\H^k$ stands for the $k$-dimensional Hausdorff measure on $\R^{n+1}$, $|\Om|$ is the Lebesgue measure (volume) of $\Om$, and $P(E)$ is the distributional perimeter of a set of finite perimeter $E\subset\R^{n+1}$ (so that $P(E)=\H^n(\pa E)$ whenever $E$ is an open set with Lipschitz boundary).

Motivated by applications to geometric variational problems (see section \ref{section capillarity intro}) we want to describe the shape of sets $\Om$ with small Alexandrov's deficit. This is a classical question in convex geometry, where the size of $\de(\Om)$ for $\Om$ convex has been related to the Hausdorff distance of $\pa\Om$ from a single sphere in various works, see \cite{schneider,arnold,kohlmann}. However one should keep in mind that, as soon as convexity is dropped off, the smallness of $\de(\Om)$ does not necessarily imply proximity to a single ball. Indeed, by slightly perturbing a given number of spheres of equal radii connected by short catenoidal necks, one can construct open sets $\{\Om_h\}_{h\in\N}$ with the property that, as $h\to\infty$, $\de(\Om_h)\to 0$, while the necks contract to points and the sets $\Om_h$ converge to an array of tangent balls, see \cite{butscher,butschermazzeo} (and, more generally for this kind of construction, the seminal papers \cite{kapouleas1990,kapouleas1991}). At the same time, if $\de(\Om)$ is small enough in terms of $n$ and the largest principal curvature of $\pa\Om$, then $\Om$ must be close to a single ball. More precisely, denoting by $A$ the second fundamental form of $\pa\Om$, in \cite{ciraolovezzoni} it is proved the existence of $c(n,\|A\|_{C^0(\pa\Om)})>0$ such that if $\de(\Om)\le c(n,\|A\|_{C^0(\pa\Om)})$, then the in-radius and out-radius of $\Om$ must satisfy
\begin{equation}
  \label{cv}
  \frac{r^{{\rm out}}(\Om)}{r^{{\rm in}}(\Om)}-1\le C(n,\|A\|_{C^0(\pa\Om)})\,\de(\Om)\,,
\end{equation}
where the linear control in terms of $\de(\Om)$ is sharp, as shown for example by taking a sequence of almost-round ellipsoids. In light of the examples from \cite{butscher}, an assumption like $\de(\Om)\le c(n,\|A\|_{C^0(\pa\Om)})$ is necessary in order to expect $\Om$ to be close to a single ball.

Our goal here is to address the situation when a different kind of smallness assumption on $\de(\Om)$ is considered. Indeed, we are just going to assume that $\de(\Om)$ is small with respect to the scale invariant quantity
\[
Q(\Om)=\frac{P(\Om)^{n+1}}{(n+1)^{n+1}|\Om|^n\,|B|}=\Big(\frac{P(\Om)}{P(B)}\Big)^{n+1}\,\Big(\frac{|B|}{|\Om|}\Big)^n\,.
\]
Notice that by the Euclidean isoperimetric inequality
\begin{equation}
  \label{euclidean iso}
  P(\Om)\ge (n+1)\,|B|^{1/(n+1)}\,|\Om|^{n/(n+1)}=P(B)\,\Big(\frac{|\Om|}{|B|}\Big)^{n/(n+1)}\,,
\end{equation}
one always has $Q(\Om)\ge 1$, and that
\[
Q\big(\mbox{a union of $L$ disjoint balls of equal radii}\big)=L\,,\qquad\forall L\in\N\,,L\ge 1\,.
\]
Hence, one may expect the integer part of $Q(\Om)$ to indicate the number of balls of radius $n/H_0$ that should be approximating $\Om$: and indeed, given $L\in\N$, $L\ge 1$, and $a\in[0,1)$, in Theorem \ref{thm main 1} we are going to prove that if $Q(\Om)\le L+1-a$ (so that the normalized perimeter of $\Om$ is a tad less than the normalized perimeter of $(L+1)$-many balls) and $\de(\Om)\le\de(n,L,a)$, then $\Om$ is close (in the various ways specified below, and quantitatively in terms of powers of $\de(\Om)$) to a compound of at most $L$-many mutually tangent balls of radius $n/H_0$.

Before stating Theorem \ref{thm main 1} it seems convenient to rescale $\Om$ in such a way that the reference balls have unit radius, that is, we rescale $\Om$ (as we can always do) in such a way that
\[
H_0=n\qquad\mbox{and thus}\qquad P(\Om)=(n+1)|\Om|\,,\qquad Q(\Om)=\frac{|\Om|}{|B|}=\frac{P(\Om)}{P(B)}\,.
\]
Here
\begin{figure}
  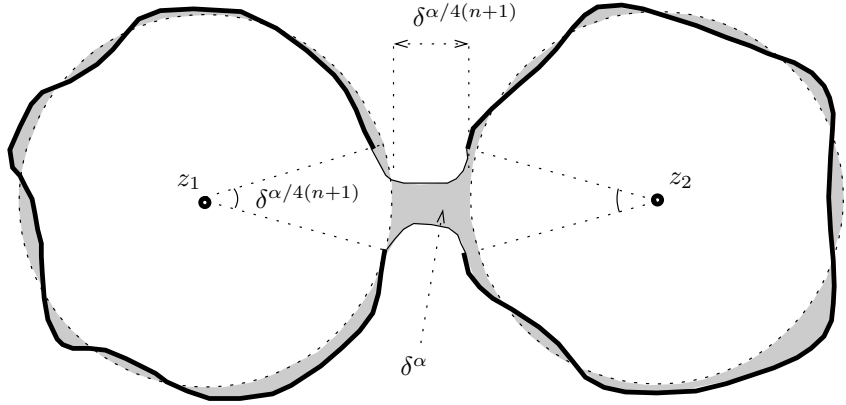\caption{{\small The situation in Theorem \ref{thm main 1}, with $\de=\de(\Om)$ and $\a=1/2(n+2)$. The grey region depicts $\Om\Delta G$ (whose area is of order $\de^\a$), while $(\Id+\psi_G\nu_G)(\Sigma)$ is depicted by a bold line. The spheres $\pa B_{z_j,1}$ are at a distance of order $\de^{\a/4(n+1)}$, while $\S$ is obtained from $\pa G$ by removing two spherical caps of diameter $\de^{\a/4(n+1)}$.
  }}\label{fig alex}
\end{figure}
and in the following we also set $B_{x,r}=\{y\in\R^{n+1}:|y-x|<r\}$ (so that $B=B_{0,1}$) and, given two compact sets $K_1, K_2$ in $\R^{n+1}$, we define their Hausdorff distance as
\[
\hd(K_1,K_2)=\max\big\{\max_{x\in K_1}\dist(x,K_2),\max_{x\in K_2}\dist(x,K_1)\big\}\,.
\]
Moreover, we let
  \begin{equation}
    \label{alpha}
     \a=\frac1{2(n+2)}\,,
  \end{equation}
and we refer readers to the beginning of section \ref{section proof of the main theorem} for our conventions about constants.

\begin{theorem}\label{thm main 1}
  Given $n,L\in\N$ with $n\ge 2$ and $L\ge 1$, and $a\in(0,1]$, there exists a positive constant $c(n,L,a)>0$ with the following property. If $\Om$ is a bounded connected open set with $C^2$-boundary in $\R^{n+1}$ such that $H>0$ and
  \[
  H_0=n\,,\qquad P(\Om)\le (L+1-a)\,P(B)\,,\qquad \de(\Om)\le c(n,L,a)\,,
  \]
  then there exists a finite family $\{B_{z_j,1}\}_{j\in J}$ of mutually disjoint balls with $\#\,J\le L$ such that if we set
  \[
  G=\bigcup_{j\in J}B_{z_j,1}\,,
  \]
  then
  \begin{eqnarray}
  \label{main thm Omega meno G}
  \frac{|\Om\Delta G|}{|\Om|}&\le&  C(n)\,L^2\,\de(\Om)^\a\,,
  \\
  \label{main thm perimetri stima}
    \frac{|P(\Om)-\#\,J\,P(B)|}{P(\Om)}\,&\le& C(n)\,L^2\,\de(\Om)^\a\,,
  \\
  \label{main thm onesided hd}
  \frac{\max_{x\in \pa G}\dist(x,\pa\Om)}{\diam(\Om)}&\le& C(n)\,L\,\de(\Om)^\a\,,
  \\\label{main thm hd stima}
    \frac{\hd(\pa\Om,\pa G)}{\diam(\Om)}&\le& C(n)\,L^{3/n}\,\de(\Om)^{\a/4n^2(n+1)}\,.
  \end{eqnarray}
  Moreover, there exists an open subset $\S$ of $\pa G$ and a function $\psi:\S\to\R$ with the following properties. The set $\pa G\setminus\S$ consists of at most $C(n)\,L$-many spherical caps whose diameters are bounded by $C(n)\,\de(\Om)^{\a/4(n+1)}$. The function $\psi$ is such that $(\Id+\psi\,\nu_G)(\S)\subset\pa\Om$ and
  \begin{gather}\label{stima psi intro}
  \|\psi\|_{C^{1,\g}(\S)}\le C(n,\g)\,,\qquad\forall\g\in(0,1)\,,
  \\
  \label{stima psi intro 2}
  \frac{\|\psi\|_{C^0(\S)}}{\diam(\Om)}\le C(n)\,L\,\de(\Om)^\a\,,\qquad \|\nabla\psi\|_{C^0(\S)}\le C(n)\,L^{2/n}\,\de(\Om)^{\a/8n(n+1)}\,,
  \\\label{stima bordo omega meno immagine di sigma}
  \frac{\H^n(\pa\Om\setminus(\Id+\psi\,\nu_G)(\S))}{P(\Om)}\le C(n)\,L^{4/n}\,\de(\Om)^{\a/4n(n+1)}\,,
  \end{gather}
  where $(\Id+\psi\,\nu_G)(x)=x+\psi(x)\,\nu_G(x)$ and $\nu_G$ is the outer unit normal to $G$. Finally:
  \begin{enumerate}
    \item[(i)] if $\#\,J\ge 2$, then for each $j\in J$ there exists $\ell\in J$, $\ell\ne j$, such that
  \begin{equation}
      \label{sjsjprimo tangenti}
    \frac{\dist(\pa B_{z_j,1},\pa B_{z_\ell,1})}{\diam(\Om)}\le C(n)\,\de(\Om)^{\a/4(n+1)}\,,
  \end{equation}
  that is to say, each ball in $\{B_{z_j,1}\}_{j\in J}$ is close to be tangent to another ball from the family;

  \item[(ii)] if there exists $\k\in(0,1)$ such that
  \begin{equation}
    \label{stima densita fuori intro}
    |B_{x,r}\setminus\Om|\ge \k\,|B|\,r^{n+1}\,,\qquad\forall x\in\pa\Om\,,r<\k\,,
  \end{equation}
  and $\de(\Om)\le c(n,L,\k)$, then $\#\,J=1$, that is, $\Om$ is close to a single ball.
  \end{enumerate}
\end{theorem}

A first consequence of Theorem \ref{thm main 1} is that examples of the kind constructed in \cite{butscher} are actually the only possible examples of boundaries with almost-constant mean curvature which are not close to a single sphere. Conversely, the examples of \cite{butscher} show that Theorem \ref{thm main 1} provides a qualitatively optimal information on sets with small Alexandrov's deficit. But of course, the strongest aspect of Theorem \ref{thm main 1} is its quantitative nature. It is precisely this last feature which is needed in order obtain explicit (although arguably non-sharp) orders of magnitude in the description of capillarity droplets, see Proposition \ref{corollary main stationary} below and the discussion after it.

A second remark is that, thanks to conclusion (i) and up to a translation of the balls $B_{z_j,1}$ of the order of $\de(\Om)$ appearing in \eqref{sjsjprimo tangenti}, one can work with a reference configuration $G$ such that $\pa G$ is connected, that is to say, for every $j\in J$ one can assert that $\pa B_{z_j,1}$ is tangent to $\pa B_{z_\ell,1}$ for some $\ell\ne j$. Of course, in doing so, the various smallness estimates \eqref{main thm Omega meno G}--\eqref{stima bordo omega meno immagine di sigma} will be of the same order of $\de(\Om)$ as in \eqref{sjsjprimo tangenti}.

The use of the constant $a$ should help to stress the ``quantization'' effect of the perimeter/energy (and of the volume) that happens under the small deficit assumption. Depending on the situation, one could be already satisfied of working with the simpler statement corresponding to the choice $a=1$.

We also comment on assumption (ii). Our idea here is to provide more robust smallness criterions for proximity to a single ball than $\de(\Om)\le\de(n,\|A\|_{C^0(\pa\Om)})$. The first criterion just amounts in asking that $P(\Om)\le (2-a)\,P(B)$, for $a\in(0,1]$. The interest of the second criterion is immediately understood if one considers that local minimizers of the capillarity energy satisfy uniform volume density estimates. Of course, a third criterion for proximity to a single ball is requiring the perimeter upper bound $P(\Om)\le 2-a$ (which corresponds to taking $L=1$).

We now illustrate the proof of Theorem \ref{thm main 1}. Our argument is based on Ros' proof of Alexandrov's theorem \cite{ros87ibero}, which follows closely the ideas of Reilly \cite{reilly}, and is based on the following {\it Heintze-Karcher inequality} \cite{heintzekarcher}: if $\Om$ is a bounded connected open set with $C^2$-boundary in $\R^{n+1}$ with $H>0$, then
\begin{equation}
    \label{hk inequality}
    \int_{\pa\Om}\frac{n}H\,d\H^n\ge(n+1)|\Om|\,.
\end{equation}
Now, if $H$ is constant, then it must be $H=H_0=n\,P(\Om)/(n+1)|\Om|$, so that $\Om$ must be an equality case in \eqref{hk inequality}. By exploiting Reilly's identity \cite{reilly}, Ros proves that if equality holds in \eqref{hk inequality}, then the solution $f$ of
\[  \left\{\begin{split}
    \Delta f=1\quad\mbox{in $\Om$}\,,
    \\
    f=0\quad\mbox{on $\pa\Om$}\,,
  \end{split}\right .
\]
satisfies $\nabla^2f=\Id/(n+1)$ on $\Om$ and $|\nabla f|=n/H_0(n+1)$ on $\pa\Om$. By $\nabla^2f=\Id/(n+1)$ on $\Om$, there exist $x_0\in\R^{n+1}$ and $c<0$ such that $f(x)=c+|x-x_0|^2/2(n+1)$ for every $x\in\Om$, i.e. $\Om$ is the ball of center $x_0$ and radius $r=\sqrt{-2(n+1)c}$, while, by $|\nabla f|=n/H_0(n+1)$ on $\pa\Om$, it must be $r=n/H_0$. When $H$ is not constant, one can still infer from the proof of \eqref{hk inequality} that
\begin{eqnarray}
\label{stima L1 hessiano f intro}
C(n)\,|\Om|\,\de(\Om)^{1/2}&\ge&\int_\Om\Big|\nabla^2 f-\frac{\Id}{n+1}\Big|\,,
\\
\label{stima L2 derivata normale intro}
C(n)\,\Big(\frac{n}{H_0}\Big)^2\,P(\Om)\,\de(\Om)&\ge&\int_{\pa\Om}\Big|\frac{n/H_0}{n+1}-|\nabla f|\Big|^2\,,
\end{eqnarray}
where the second estimate holds if $\de(\Om)\le1/2$, and where $\nabla f=|\nabla f|\,\nu_\Om\ne 0$ on $\pa\Om$.

The problem of exploiting \eqref{stima L1 hessiano f intro} and \eqref{stima L2 derivata normale intro} in the description of $\Om$ has some analogies with the quantitative analysis of Serrin's overdetermined problem \cite{serrinrigidity} addressed in \cite{brandolininitschsalanitrombetti}. In our terminology, the main result from \cite{brandolininitschsalanitrombetti} states that, if $H_0=n$ and for some $t>0$ one has
\begin{equation}
  \label{salanitsch hp}
  \int_{\pa\Om}\,\Big|\frac{1}{n+1}-|\nabla f|\Big|\le  P(\Om)\,t\,,\qquad \|\nabla f\|_{C^0(\pa\Om)}\le \frac{1+t}{n+1}\,,
\end{equation}
then there exist finitely many disjoint balls $\{B_{x_i,r_i}\}_{i=1}^m$ such that
\begin{eqnarray}\label{salanitsch tesi}
  \Big|\Om\Delta\bigcup_{i=1}^mB_{x_i,r_i}\Big|^{(n+1)/2}+\max_{1\le i\le m}\big|r_i-1\big|
  \le C(n,\diam(\Om))\,t^{\beta}\,,\qquad \beta=\frac1{4n+13}\,.
\end{eqnarray}
Because of the uniform upper bound on $|\nabla f|$ in \eqref{salanitsch hp}, it is not clear if one can take advantage of this result in the proof of Theorem \ref{thm main 1}. At the same time, we have a different condition at our disposal, namely \eqref{stima L1 hessiano f intro}, and by combining \eqref{stima L1 hessiano f intro} with a global Lipschitz estimate for $f$ (which is based on \cite{caffarelligarafolosegala}, and exploits the geometric assumption that $H>0$ on $\pa\Om$), we obtain a more precise control than \eqref{salanitsch tesi} on the distance of $\Om$ from a finite family of balls. (Indeed, the power $\a$ in \eqref{main thm Omega meno G} is larger than the power $2\b/(n+1)$ appearing in \eqref{salanitsch tesi}.) Finally, in Serrin's overdetermined problem the limiting balls need not to be tangent (sets $\Omega$ with small $t$ may contain arbitrarily long connecting necks) and one does not expect Hausdorff estimates like \eqref{main thm onesided hd} and \eqref{main thm hd stima} to hold (a set $\Om$ with small $t$ may contain small inclusions of large mean curvature). In other words, although \eqref{salanitsch tesi} provides a qualitatively sharp information in the context of Serrin's problem, in the case of Alexandrov's theorem one expects, and thus wants to obtain, stronger information on $\Om$.

Coming back to the proof of Theorem \ref{thm main 1}, the first step consists in proving a qualitative result, see Theorem \ref{thm compattezza} below. Indeed, by combining a compactness argument, \eqref{stima L1 hessiano f intro} and \eqref{stima L2 derivata normale intro} with Reilly's identity, Pohozaev's identity, and Allard's regularity theorem (for integer rectifiable varifolds with bounded distributional mean curvature) one comes to prove the following fact: if $\{\Om_h\}_{h\in\N}$ is a sequence of open, bounded and connected sets with $C^2$-boundary in $\R^{n+1}$, $n\ge 2$, such that for some $L\in\N$, $L\ge 1$,
\[
\lim_{h\to\infty}\de(\Om_h)=0\,,\qquad\sup_{h\in\N}Q(\Om_h)<L+1\,,
\]
then, setting
\[
\l_h=\frac{P(\Om_h)}{(n+1)|\Om_h|}\,,\qquad\Om_h^*=\l_h\,\Om_h\,,
\]
and up to translations, one has
\[
\lim_{h\to\infty}\hd(\pa\Om_h^*,\pa G)+|P(\Om_h^*)-P(G)|=0\,,
\]
where $G$ is the union of at most $L$-many disjoint balls with unit radii, and with $\pa G$ connected. Moreover, for every $h$ large enough there exist open sets $\S_h\subset\pa G$ (obtained by removing from $\pa G$ at most $C(n)\,L$-many spherical caps) and functions  $\psi_h\in C^{1,\g}(\S_h)$ for every $\g\in(0,1)$ such that $(\Id+\psi_h\,\nu_G)(\S_h)\subset\pa\Om_h^*$, and
\[
\lim_{h\to\infty}\hd(\S_h,\pa G)=0\,,\qquad \|\psi_h\|_{C^{1,\g}(\S_h)}\le C(n,\g)\,,\qquad \lim_{h\to\infty}\|\psi_h\|_{C^1(\S_h)}=0\,.
\]
This qualitative stability result, Theorem \ref{thm compattezza}, is not needed in the proof of Theorem \ref{thm main 1}, and of course it is actually a corollary of it. We have nevertheless opted for including a direct discussion of it for the following reasons. First of all, it is a result of independent interest and possible usefulness, so it seems interesting to have a shorter proof of it. Secondly, by having Theorem \ref{thm compattezza} at hand one is able to clean up to some later quantitative arguments and obtain better estimates. Thirdly, Theorem \ref{thm main 1} is actually proved by quantitatively revisiting the proof of Theorem \ref{thm compattezza}, and therefore the separate treatment of the latter should makes more accessible the argument used in proving the former.

In this direction the main difficulty arises in the application of the area excess regularity criterion of Allard, which is needed to parameterize a large portion of $\pa\Om$ over a large portion of $\pa G$. A key point here is quantifying the size of $\H^n(\pa\Om\cap B_{x,r})$ on a range of scales $r$ proportional to a suitable power of $\de(\Om)$ and at points $x\in\pa\Om$ sufficiently close to $\pa G$. We address this issue by carefully partitioning $\R^{n+1}$ into suitable polyhedral regions associated to the balls $B_{z_j,1}$, and by then performing inside each of these regions a calibration type argument with respect to the corresponding ball $B_{z_j,1}$ (see, in particular, step six of the proof of Theorem \ref{thm main 1}).

Summarizing, Theorem \ref{thm main 1} is proved by combining a mix of different ideas from elliptic PDE theory, global geometric identities, and geometric measure theory, and it contains a quantitative (and qualitatively sharp) description of boundaries with almost-constant mean curvature.

\subsection{An application to capillarity surfaces}\label{section capillarity intro} The study of the basic capillarity-type energy functional in $\R^{n+1}$ leads to consider sets $\Om$ with small Alexandrov's deficit. Indeed, given a potential energy density $g:\R^{n+1}\to\R$, in capillarity theory one considers the free energy
\begin{equation*} 
  \F(\Om)=P(\Om)+\int_\Om\,g(x)\,dx\,,
\end{equation*}
and its volume-constrained stationary points and local/global minimizers. Capillarity phenomena are characterized by the dominance of surface tension over potential energy, which is the case when the volume parameter $m=|\Om|$ is small (as surface tension is of order $m^{n/(n+1)}$, while potential energy is typically of order $m$). Under mild assumptions on $g$ (essentially, coercivity at infinity, $g(x)\to\infty$ as $|x|\to\infty$), one can show the existence of volume-constrained global minimizers of $\F$ of any fixed volume. In particular, if $\Om_m$ is such a global minimizer with $|\Om_m|=m$, then by comparison with a ball $B^{(m)}$ of volume $m$ one sees that
\[
P(\Om_m)\le P(B^{(m)})+\int_{B^{(m)}\setminus\Om_m}\,g(x)\,dx\,,
\]
that is, the {\it isoperimetric deficit} $\de_{{\rm iso}}(\Om_m)$ of $\Om_m$ is small in terms of $m$,
\[
\de_{{\rm iso}}(\Om_m)=\frac{P(\Om_m)}{P(B^{(m)})}-1\le \frac{C(n)}{m^{n/(n+1)}}\,\int_{B^{(m)}\setminus\Om_m}\,g(x)\,dx\approx C(n,g)\,m^{1/(n+1)}\,.
\]
By the quantitative isoperimetric inequality \cite{fuscomaggipratelli,FigalliMaggiPratelliINVENTIONES}, one finds $x_m\in\R^{n+1}$ such that
\[
\Big(\frac{|\Om_m\Delta (x_m+B^{(m)})|}m\Big)^2\le C(n)\,\de_{{\rm iso}}(\Om_m)\,,
\]
so that, in conclusion, $\Om_m$ has to be close (in a normalized $L^1$-sense) to a ball of volume $m$. This observation is the starting point of the analysis performed in \cite{FigalliMaggiARMA}, where the proximity of $\Om_m$ to a ball of volume $m$ is quantified, under increasingly stronger smoothness assumptions on $g$, in increasingly stronger ways. For example, if $g\in C^1_{{\rm loc}}(\R^{n+1})$ and $m\le m_0(n,g)$, then $\Om_m$ is shown to be convex and $\pa\Om_m$ is proved to a $C^{2,\g}$-small normal deformation of $x_m+\pa B^{(m)}$, with explicit quantitative bounds on the $C^{2,\g}$-norm of this deformation in terms of $m$.

When dealing with volume-constrained local minimizers or stationary points of $\F$ one cannot rely anymore on the quantitative isoperimetric inequality, as one is not given the energy comparison inequality with $B^{(m)}$. However, in this more general context, Alexandrov's deficit turns out to be small in terms of the volume parameter $m$, thus opening the way for the application of Theorem \ref{thm main 1}.

Let us recall that given a vector field $X\in C^\infty_c(\R^{n+1};\R^{n+1})$, and denoted by $f_t$ the flow generated by $X$, then the first variation of $\F$ at $\Om$ along $X$ is defined as
\begin{equation}
  \label{stationary set intro}
  \de\F(\Om)[X]=\frac{d}{dt}\Big|_{t=0}\,\F(f_t(\Om))\,.
\end{equation}
One says that a set of finite perimeter $\Om\subset\R^{n+1}$ is a {\it volume-constrained stationary point of $\F$} if $\de\F(\Om)[X]=0$ for every $X\in C^\infty_c(\R^{n+1};\R^{n+1})$ such that $|f_t(\Om)|=|\Om|$ for every $t$ small enough. The following proposition, combined with Theorem \ref{thm main 1}, provides a complete description of such stationary boundaries, and its simple proof is presented in section \ref{section capillarity}.

\begin{proposition}\label{corollary main stationary}
  Let $g\in C^1_{{\rm loc}}(\R^{n+1})$, $R_0>0$, and $\Om$ be an open set with $C^2$-boundary such that $\Om\subset B_{R_0}$. If $\Om$ is a volume-constrained stationary point of $\F$ with $|\Om|=m$, then
  \begin{equation}
    \label{delta stazionario}
      \de(\Om)\le C_*(n)\,\|g\|_{C^1(B_{R_0})}\,m^{1/(n+1)}\,,
  \end{equation}
  for some constant $C_*(n)$.
\end{proposition}

Under the assumptions of Proposition \ref{corollary main stationary}, let us now pick $L\in\N$, $L\ge 1$, and $a\in(0,1]$, define $c(n,L,a)$ as in Theorem \ref{thm main 1}, and assume that
\begin{eqnarray*}
Q(\Om)\le L+1-a\,,\qquad m\le \Big(\frac{c(n,L,a)}{C_*(n)\,\|g\|_{C^1(B_{R_0})}}\Big)^{n+1}\,.
\end{eqnarray*}
In this way, by Theorem \ref{thm main 1} and \eqref{delta stazionario}, there exists a finite family $\{B_{z_j,1}\}_{j\in J}$ of disjoint balls such that, looking for example at \eqref{main thm Omega meno G} and setting $G=\bigcup_{j\in J}B_{z_j,1}$,
\[
\frac{|\Om\Delta G_*|}{|\Om|}=\frac{|\Om^*\Delta G|}{|\Om^*|}\le C(n)\,L^2\,\|g\|_{C^1(B_{R_0})}^\a\,m^{\a/(n+1)}\,,
\]
where $\Om^*=(H_0/n)\Om$, and thus $G_*=\bigcup_{j\in J}B_{w_j,n/H_0}$. Notice that this proves a quantization of the volume of $\Om$, in the sense that $|\Om^*|$ is close to $\#J\,|B|$ with an error of order $C\,m^{\a/(n+1)}$, where $C=C(n,L,\|g\|_{C^1(B_{R_0})})$. Similar results carrying different geometric information are obtained from the other estimates appearing in Theorem \ref{thm main 1}.

In conclusion, the quantitative side of Theorem \ref{thm main 1} significantly strengthens the purely qualitative analysis that one could obtain by exploiting compactness arguments only, as it provides explicit orders of magnitude for the errors one makes in approximating $\Om^*$ with a unit balls compound.

We finally notice that $\Om^*$ will be close to a single ball as soon as volume-constrained stationarity is strengthened into some {\it local} minimality property. For example, it will suffice to require that $\F(\Om)\le\F(E)$ whenever $|E|=|\Om|$ and $\pa E\subset I_\s(\pa\Om)=\{x\in\R^{n+1}:\dist(x,\pa\Om)<\s\}$, with $\s=\s_0\,|\Om|/P(\Om)$ for some $\s_0>0$. Notice that although $E=B^{(m)}$ is not an admissible competitor in this local minimality condition, thus ruling out the possibility of applying \cite{fuscomaggipratelli}, $\Om$ will nevertheless be a volume-constrained stationary set for $\F$. Moreover, by a standard argument exploiting the local minimality of $\Om$, one obtains volume density estimates for $\Om^*$ which make possible to apply statement (ii) in Theorem \ref{thm main 1}.

\subsection{Organization of the paper} The proof of Theorem \ref{thm main 1} and of Proposition \ref{corollary main stationary} are discussed, respectively, in section \ref{section proof of the main theorem} and section \ref{section capillarity}. In Appendix \ref{appendix umbilical} we discuss the relation of the Alexandrov's stability problem with the study of almost-umbilical surfaces initiated by De Lellis and M\"uller in \cite{delellismuller1}.

\medskip

\noindent {\bf Acknowledgment:} We thank Manuel Ritor\'e for stimulating the writing of Appendix \ref{appendix umbilical}. This work has been done while GC was visiting the University of Texas at Austin under the support of NSF-DMS FRG Grant 1361122, of a Oden Fellowship at ICES, the GNAMPA of the Istituto Nazionale di Alta Matematica (INdAM), and the FIRB project 2013 ``Geometrical and Qualitative aspects of PDE''. FM is supported by NSF-DMS Grant 1265910 and  NSF-DMS FRG Grant 1361122.

\section{Proof of Theorem \ref{thm main 1}}\label{section proof of the main theorem} We begin by gathering various assumptions, preliminaries, and conventions.

\medskip

\noindent {\bf Constants}: The symbol $C$ denotes a generic positive constant whose value is independent from $n$ and $\Om$. We use the symbols $C_0$, $C_1$, etc. for constants whose specific value is referred to in multiple occasions (see, for instance \eqref{f Lip} below). We denote by $C(n)$ and $c(n)$ generic positive constants whose value does depend on $n$, but is independent from $\Om$, with the idea that $C(n)$ stands for a ``large'' constant, and $c(n)$ stands for a ``small'' constant. Similar conventions hold for $C(n,L)$, etc.


\medskip

\noindent {\bf Assumptions on $\Om$}: Thorough this section we always assume that
\begin{equation}\label{Omega C2 con H positiva}
  \begin{split}
    &\mbox{$\Om\subset\R^{n+1}$, $n\ge 2$, is a bounded connected open set}
    \\
    &\mbox{with $C^2$-boundary with $H>0$ on $\pa\Om$}\,.
  \end{split}
\end{equation}
From a certain point of our argument on we shall assume that (as one can always do up to a scaling)
\begin{equation}\label{Omega riscalato H0 uguale n}
  (n+1)|\Om|=P(\Om)\,.
\end{equation}
Recall that, by the Euclidean isoperimetric inequality (see \eqref{euclidean iso}), \eqref{Omega riscalato H0 uguale n} implies
\begin{equation}
  \label{Omega volume e perimetro maggiori B}
  |B|\le|\Om|\,,\qquad P(B)\le P(\Om)\,,
\end{equation}
where $B=\{x\in\R^n:|x|<1\}$. Moreover, \eqref{Omega riscalato H0 uguale n} is equivalent to $H_0=n$, so that
\[
\de(\Om)=n^{-1}\,\|H-n\|_{C^0(\Om)}\,,
\]
with $\de(\Om)=0$ if and only if $\Om$ is a ball of unit radius. Indeed, our convention for the scalar mean curvature $H$ is that
\[
\int_{\pa\Om}\,\Div^{\pa\Om}\,X\,d\H^n=\int_{\pa\Om}(X\cdot\nu_\Om)\,H\,d\H^n\,,\qquad\forall X\in C^1_c(\R^{n+1};\R^{n+1})\,,
\]
and thus the scalar mean curvature of $B$ is equal to $n$. In addition to \eqref{Omega riscalato H0 uguale n} we  also assume that
\begin{equation}
  \label{Omega perimetro minore L+1}
  P(\Om)\le(L+1-a)P(B)\,,\qquad\mbox{where $L\in\N$, $L\ge 1$, $a\in(0,1]$.}
\end{equation}
Note that, by combining \eqref{Omega perimetro minore L+1} with \eqref{Omega riscalato H0 uguale n} one finds
\begin{equation}
  \label{Omega volume minore L+1}
  |\Om|\le(L+1-a)|B|\,.
\end{equation}
We shall work under the assumption that $\de(\Om)\le c(n,L,a)$ for a suitably small positive constant $c(n,L,a)\le 1/2$: in particular,
\begin{equation}
  \label{Omega H fra nmezzi e 2n}
  \frac{n}2\le H(x)\le 2n\,,\qquad\forall x\in\pa\Om\,.
\end{equation}
By Topping's inequality \cite{topping}, one has
\[
\diam(\Om)\le C(n)\,\int_{\pa\Om}\,|H|^{n-1}\,,
\]
so that \eqref{Omega perimetro minore L+1} and \eqref{Omega H fra nmezzi e 2n} imply
\begin{eqnarray}
  \label{Omega diametro minore L}
  \diam(\Om)\le C(n)\,L\,.
\end{eqnarray}
Alternatively, by the monotonicity identity (see \cite[Theorem 2.1]{DeLellisNOTES}) and by $H\le 2n$ on $\pa\Om$, one has that
\begin{equation}
  \label{monotonicity formula}
  s\in(0,\infty)\mapsto e^{2ns}\,\frac{\H^n(\pa\Om\cap B_{x,s})}{s^n}\qquad\mbox{is monotone increasing for every $x\in\R^{n+1}$}\,.
\end{equation}
If $x\in\pa\Om$, then this function converges to $\H^n(\{z\in\R^n:|z|<1\})$ as $s\to 0^+$, and thus one obtains the uniform lower perimeter estimate
\begin{equation}
  \label{Omega stima densita perimetro basso}
  \H^n(\pa\Om\cap B_{x,s})\ge c(n)\,s^n\,,\qquad\forall x\in\pa\Om\,,s\in(0,1)\,.
\end{equation}
We notice that this last fact can be used jointly with $P(\Om)\le C(n)\,L$ to infer \eqref{Omega diametro minore L}. Finally, whenever $\Om$ satisfies \eqref{Omega C2 con H positiva} we define the {\it Heintze-Karcher deficit of $\Om$} as
\begin{equation}
  \label{def eta}
    \eta(\Om)=\frac{\int_{\pa\Om}\frac{n}H-(n+1)|\Om|}{\int_{\pa\Om}\frac{n}H}=1-\frac{(n+1)|\Om|}{\int_{\pa\Om}\frac{n}H}\,.
\end{equation}
Just like $\de(\Om)$, this is a scale invariant quantity such that $\eta(\Om)=0$ if and only if $\Om$ is a ball. One has
\begin{equation}
  \label{eta Omega minore delta Omega}
  \eta(\Om)\le\de(\Om)\,.
\end{equation}
Indeed,
  \begin{eqnarray*}
  \eta(\Om)&=&1-\frac{(n+1)|\Om|}{\int_{\pa\Om}\frac{n}H}=\frac{(n+1)|\Om|}{\int_{\pa\Om}\frac{n}H_0}-\frac{(n+1)|\Om|}{\int_{\pa\Om}\frac{n}H}
  \\
  &=&\frac{(n+1)|\Om|}n\,\frac{\int_{\pa\Om}\frac{1}H-\int_{\pa\Om}\frac{1}H_0}{\int_{\pa\Om}\frac{1}H_0\int_{\pa\Om}\frac{1}H}
  \le\frac{(n+1)|\Om|}n\,\frac{\de(\Om)\,\int_{\pa\Om}\frac{1}H}{\int_{\pa\Om}\frac{1}H_0\int_{\pa\Om}\frac{1}H}=\de(\Om)\,.
\end{eqnarray*}

\medskip

\noindent {\bf Torsion potential}: We denote by $f$ and $u$ the smooth functions defined on $\Om$ by setting
\begin{equation}
  \label{f definizione}
  \left\{\begin{split}
    \Delta f=1\quad\mbox{in $\Om$}\,,
    \\
    f=0\quad\mbox{on $\pa\Om$}\,,
  \end{split}\right .\qquad u=-f\,.
\end{equation}
Note that $f<0$ on $\Om$, with $\nabla f=|\nabla f|\,\nu_\Om$ on $\pa\Om$, and $\nabla_\nu f=\nu_\Om\cdot\nabla f=|\nabla f|>0$ on $\pa\Om$ by Hopf's lemma. We shall use two integral identities involving $f$, namely, the {\it Reilly's identity} (see, e.g., \cite[Equation (3)]{ros87ibero})
  \begin{equation}\label{reilly identity f constant}
  \int_{\pa\Om}\,H\,|\nabla f|^2=\int_\Om (\Delta f)^2-|\nabla^2f|^2\,,
  \end{equation}
and the {\it Pohozaev's identity}, see e.g. \cite[Theorem 8.30]{ambrosettimalchiodi},
\begin{equation}
  \label{pohozaev}
  (n+3)\int_\Om\,(-f)=\int_{\pa\Om}(x\cdot\nu_\Om)|\nabla f|^2\,.
\end{equation}
The first one quickly leads to prove Alexandrov's theorem and the Heintze-Karcher inequality, as shown in \cite{ros87ibero}.

\begin{lemma}\label{lemma hk}
  If $\Om$ and $f$ are as in \eqref{Omega C2 con H positiva} and \eqref{f definizione}, then
  \begin{eqnarray}\label{ros identity}
    \frac{|\Om|}{n+1}\,\Big(\int_{\pa\Om}\frac{n}H\,-(n+1)|\Om|\Big)
    &=&\int_{\pa\Om}\frac{1}{H}\,\int_\Om\,|\nabla^2f|^2-\frac{(\Delta f)^2}{n+1}
    \\\nonumber
    &&+\int_{\pa\Om}\frac{1}{H}\,\int_{\pa\Om}|\nabla f|^2\,H\,-\Big(\int_{\pa\Om}|\nabla f|\Big)^2\,,
  \end{eqnarray}
  In particular, \eqref{hk inequality} holds, and if $H$ is constant on $\pa\Om$, then $H=H_0>0$ and $\Om$ is a ball.
\end{lemma}

\begin{proof} By the divergence theorem and by H\"older's inequality,
  \begin{eqnarray*}
  |\Om|^2=\big(\int_{\pa\Om}\nabla_\nu f\,\big)^2=\Big(\int_{\pa\Om} \frac{\sqrt{H}\,|\nabla f|}{\sqrt{H}}\Big)^2
  \le \int_{\pa\Om}\,\frac{1}{H}\,\int_{\pa\Om} |\nabla f|^2\,H\,.
  \end{eqnarray*}
  Thanks to \eqref{reilly identity f constant},
  \[
  \int_{\pa\Om}\,H\,|\nabla f|^2=\frac{n}{n+1}\,|\Om|+\int_\Om \frac{(\Delta f)^2}{n+1}-|\nabla^2f|^2\le \frac{n}{n+1}\,|\Om|\,,
  \]
  where we have used the Cauchy-Schwartz inequality
  \[
  (\trace M)^2=(M:\Id)^2\le|\Id|^2\,|M|^2=(n+1)\,|M|^2\,,\qquad\forall M\in\R^n\otimes\R^n\,.
  \]
  (Here and in the following, we denote by $:$ the scalar product on $\R^n\otimes\R^n$, and by $|\cdot|$ the corresponding Hilbert norm on $\R^n\otimes\R^n$.) This proves \eqref{ros identity}. Let us now assume that $H$ is constant on $\pa\Om$, then by applying the divergence theorem on $\pa\Om$ and on $\Om$, one finds
  \begin{eqnarray}\label{computation}
  \int_{\pa\Om}\frac{n}H&=&\frac{n P(\Om)}{H}=\frac{\int_{\pa\Om}\Div^{\pa\Om}(x)\,d\H^n_x}{H}
  \\\nonumber
  &=&\frac{\int_{\pa\Om}(x\cdot\nu_\Om)\,H\,d\H^n_x}{H}=\int_{\pa\Om}x\cdot\nu_\Om\,d\H^n_x=\int_\Om\,\Div(x)\,dx=(n+1)|\Om|\,,
  \end{eqnarray}
  so that $H=H_0$ and equality holds in \eqref{hk inequality}. In particular, \eqref{ros identity} gives that $\nabla^2f(x)=\Id/(n+1)$ for every $x\in\Om$, so that, being $\Om$ connected,
  \[
  f(x)=c+\frac{|x-x_0|^2}{2(n+1)}\,,
  \]
  for some $c<0$ and $x_0\in\R^n$. Since $f=0$ on $\pa\Om$, we find that $\Om$ is the ball of center $x_0$ and radius $\sqrt{-2(n+1)c}$.
\end{proof}

We now exploit \cite{talenti1976} and \cite{caffarelligarafolosegala} to obtain universal estimates on $f$.

\begin{lemma}\label{lemma p function}
  If $\Om$ and $f$ are as in \eqref{Omega C2 con H positiva} and \eqref{f definizione}, then
  \begin{eqnarray}
    \label{f C0}
    \|f\|_{C^0(\Om)}&\le& \frac{1}{2(n+1)}\,\Big(\frac{|\Om|}{|B|}\Big)^{2/(n+1)}\le C\,|\Om|^{2/(n+1)}\,,
    \\
    \label{f Lip}
    \|\nabla f\|_{C^0(\Om)}&\le& \sqrt{2}\,\|f\|_{C^0(\Om)}^{1/2}\le C_0\,|\Om|^{1/(n+1)}\,,
    \\
    \label{f L2 estimates}
    \|\nabla^2f\|_{L^2(\Om)}&\le&|\Om|^{1/2}\,.
  \end{eqnarray}
\end{lemma}

\begin{proof}
  By a classical result of Talenti \cite{talenti1976}, the radially symmetric decreasing rearrangement $(-f)^\star$ of $-f$ satisfies the pointwise estimate
  \begin{equation}
    \label{pointwise talenti}
      (-f)^\star(x)\le \frac{R^2-|x|^2}{2(n+1)}\,,\qquad\mbox{where}\qquad R=\Big(\frac{|\Om|}{|B|}\Big)^{1/(n+1)}\,,
  \end{equation}
  so that the first inequality in \eqref{f C0} follows immediately. The second inequality in \eqref{f C0} is then obtained by recalling that
  \begin{equation}
    \label{palle e gamma}
      |\{z\in\R^k:|z|<1\}|=\frac{\pi^{k/2}}{\Gamma(1+(k/2))}\,,\qquad \lim_{t\to\infty}\frac{\Gamma(1+t)}{\sqrt{2\pi t}(t/e)^t}=1\,.
  \end{equation}
  (Thus \eqref{f C0} does not need the assumption that $H>0$ on $\pa\Om$). Moreover, we immediately deduce \eqref{f L2 estimates} from $\Delta f=1$ and \eqref{reilly identity f constant}, so that we are left to prove \eqref{f Lip}. With $u=-f$ we set
  \[
  p=|\nabla u|^2+2(u-\|u\|_{C^0(\Om)})\,,
  \]
  and aim to prove that $p\le 0$. The key fact is the observation that
  \begin{equation}
    \label{grande luis}
      |\nabla u|^2\,\Delta p+2\,\nabla u\cdot\nabla p\ge \frac{|\nabla p|^2}2\,,\qquad\,\mbox{on $\{|\nabla u|>0\}$}\,,
  \end{equation}
  see \cite[Equation (2.7)]{caffarelligarafolosegala}. Given \eqref{grande luis}, we argue by contradiction and assume that the maximum $p_0$ of $p$ in $ \ov{\Om}$ is positive. We first claim that $p_0$ is achieved on $\pa\Om$. Indeed, let $U=\{x\in\Om:p(x)=p_0\}$, then $U$ is obviously closed. If $x\in U$, then $|\nabla u(x)|^2\ge p(x)=p_0>0$ and so $p$ satisfies
  \[
  \Delta p+T\cdot\nabla p\ge 0\,,\qquad\mbox{in a neighborhood of $x$}\,,
  \]
  where the vector-field
  \begin{equation}\label{T_def}
  T=2\frac{\nabla u}{|\nabla u|^2}
  \end{equation}
  is bounded on that same neighborhood. By the strong maximum principle, $p$ must be constant in that neighborhood. This shows that $U$ is open, so that $U=\Om$ by connectedness. At the same time, there exists $x^*\in\Om$ such that $\nabla u(x^*)=0$, so that $p_0=p(x^*)=2\big(u(x^*)-\|u\|_{C^0(\Om)}\big)\le 0$ a contradiction. This shows that there exists $x_0\in\pa\Om$ such that
  \[
  p(x_0)=p_0>p(x)\,,\qquad\forall x\in\Om\,.
  \]
  By Hopf's lemma, $\nabla_\nu u(x_0)<0$, so that
  \[
  \Delta p+T\cdot\nabla p\ge 0\,,\qquad\mbox{in a neighborhood of $x_0$ in $\Om$}\,,
  \]
  where once again the vector-field $T$ (defined as in \eqref{T_def} above) is bounded. By Hopf's lemma,
  \begin{equation}
    \label{fv1}
      \nabla_\nu p(x_0)>0\,.
  \end{equation}
  At the same time one has
  \[
  \nabla_\nu p(x_0)=2\Big( \nabla u(x_0)\cdot\nabla(\nabla_\nu u)(x_0)+\nabla_\nu u(x_0) \Big)\,.
  \]
  Since $u=0$ on $\pa\Om$, we have $\nabla u=(\nabla_\nu u)\,\nu$ on $\pa\Om$, so that the above identity becomes
  \[
  \nabla_\nu p(x_0)=\nabla_\nu u(x_0)\big(\nabla_{\nu\,\nu} u(x_0)+1\big)\,.
  \]
  Since $\nabla_\nu u(x_0)<0$, \eqref{fv1} gives us
  \begin{equation}
    \label{fv2}
    \nabla_{\nu\,\nu} u(x_0)<-1=\Delta u(x_0)\,.
  \end{equation}
  We now obtain a contradiction by showing that, thanks to $H>0$, one has
  \begin{equation}\label{fv3}
    \Delta u(x_0)<\nabla_{\nu\,\nu} u(x_0)\,.
  \end{equation}
  Indeed, assuming without loss of generality that $x_0=0$ and that $\Om$ is (locally at $0$) the subgraph of a function $\vphi$ on $n$-variables such that $\vphi(0)=0$ and $\nabla\vphi(0)=0$ (so that $\nu_\Om(0)=e_n$, and thus $-H(0)=\Delta\vphi(0)$), by differentiating $u(z,\vphi(z))=0$ at $z=0$ twice along the direction $z_i$, one gets
  \[
  0=\nabla_{z_iz_i}u(z,\vphi(z))+2\nabla_{z_i\,\nu}u(z,\vphi(z))\,\nabla_{z_i}\vphi(z)+\nabla_\nu u(z,\vphi(z))\,\nabla_{z_i\,z_i}\vphi(z)+\nabla_{\nu\nu}u(z)\,(\nabla_{z_i}\vphi(z))^2
  \]
  which, evaluated at $z=0$, by $\nabla\vphi(0)=0$ gives us
  \[
  0=\nabla_{z_iz_i}u(0)+\nabla_\nu u(0)\,\nabla_{z_i\,z_i}\vphi(0)\,.
  \]
  By adding up over $i=1,...,n$, and by $H(0)>0$ and $\nabla_\nu u(0)<0$, we conclude that
  \[
  0=\Delta u(0)-\nabla_{\nu\,\nu}u(0)-H(0)\,\nabla_\nu u(0)>\Delta u(0)-\nabla_{\nu\,\nu}u(0)\,,
  \]
  so that \eqref{fv3} holds.
\end{proof}

\begin{lemma}\label{lemma stime integrali}
  If $\Om$ and $f$ are as in \eqref{Omega C2 con H positiva} and \eqref{f definizione}, then
  \begin{equation}
    \label{stima L1 hessiano f}
      C(n)\,|\Om|\,\sqrt{\eta(\Om)}\ge \int_\Om\Big|\nabla^2 f-\frac{\Id}{n+1}\Big|\,.
  \end{equation}
  If, in addition, $\de(\Om)\le 1/2$, then
    \begin{equation}
  \label{stima L2 derivata normale}
  C(n)\,\Big(\frac{n}{H_0}\Big)^2\, P(\Om)\,\de(\Om)\ge \int_{\pa\Om} \Big|\frac{n/H_0}{n+1}-|\nabla f|\Big|^2\,.
\end{equation}
\end{lemma}

\begin{remark}\label{remark D2u meno mu}
  {\rm If we define $\bar u=-f$ on $\Om$, $\bar u=0$ on $\R^{n+1}\setminus\Om$, then the distributional gradient $D\bar u$ and the distributional Hessian $D^2\bar u$ of $\bar u$ are given by
  \begin{eqnarray*}
  D\bar u&=&-\nabla f\,\Le^{n+1}\llcorner\Om\,,
  \\
  D^2\bar u&=&-\nabla^2f\,\Le^{n+1}\llcorner\Om+\frac{\nabla f\otimes\nabla f}{|\nabla f|}\,\H^n\llcorner\pa\Om\,,
  \end{eqnarray*}
  where $\Le^{n+1}$ is the Lebesgue measure on $\R^{n+1}$. Indeed, for every $\vphi\in C^\infty_c(\R^{n+1})$ one has
  \[
    D^2\bar u(\vphi)=\int_{\R^{n+1}}\bar u\,\pa_{ij}\vphi=-\int_{\Om}\,f\,\pa_{ij}\vphi=\int_{\Om}\,\pa_if\,\pa_{j}\vphi=
    \int_{\pa\Om}\,(\nu_{\Om})_j\,\vphi\,\pa_if-\int_{\Om}\vphi\,\pa_{ij}f\,,
  \]
  where $\nu_{\Om}=\nabla f/|\nabla f|$ on $\pa\Om$. Hence, under the assumption \eqref{Omega riscalato H0 uguale n}, \eqref{stima L1 hessiano f} and \eqref{stima L2 derivata normale} are equivalent to
  \begin{equation}
    \label{boh}
      |D^2\bar u-\mu|(\R^{n+1})\le C(n)\,\big(P(\Om)\,\de(\Om)+|\Om|\,\eta(\Om)^{1/2}\big)\le C(n,L)\,\de(\Om)^{1/2}\,,
  \end{equation}
  where $\mu$ is the Radon measure defined by
  \[
  \mu=-\frac{\Id}{n+1}\,\Le^{n+1}\llcorner\Om+\frac{\nu_\Om\otimes\nu_\Om}{n+1}\,\H^n\llcorner\pa\Om\,.
  \]
  This point of view on \eqref{stima L1 hessiano f}--\eqref{stima L2 derivata normale} is at the basis of the proof of Theorem \ref{thm compattezza} below.
  }
\end{remark}


\begin{proof}[Proof of Lemma \ref{lemma stime integrali}] We first prove \eqref{stima L1 hessiano f}. If $M_1,M_2\in\R^n\otimes\R^n$ with $M_1,M_2\ne 0$, then one has
  \[
  |M_1|\,|M_2|-M_1:M_2=\frac12\Big|\mu\,M_1-\frac{M_2}\mu\Big|^2\,,\qquad\mu=(|M_2|/|M_1|)^{1/2}\,,
  \]
  so that
  \[
  |M_1|^2\,|M_2|^2-(M_1:M_2)^2\ge (M_1:M_2)\,\Big|\mu\,M_1-\frac{M_2}\mu\Big|^2\,.
  \]
  By \eqref{ros identity}, setting $M_1=\nabla^2f$ (note that $\nabla^2f\ne 0$ as $\Delta f=1$ on $\Om$) and $M_2=\Id$, and noticing that $|\Id|^2=(n+1)$ and $\Delta f=\nabla^2f:\Id$, one finds
  \begin{eqnarray}\label{stima hessiano 1}
    n\,|\Om|\,\eta(\Om)\ge
    \int_\Om\,|\Id|^2\,|\nabla^2f|^2-(\Delta f)^2\ge\int_\Om\,\mu^2\,\Big|\nabla^2f-\frac{\Id}{\mu^2}\Big|^2\,,
  \end{eqnarray}
  where we have set $\mu(x)=\big(|\Id|/|\nabla^2f(x)|\big)^{1/2}$, $x\in\Om$. By \eqref{f L2 estimates} and \eqref{stima hessiano 1}, we get
  \begin{eqnarray*}
    \Big(\int_\Om\Big|\nabla^2f-\frac{\Id}{\mu^2}\Big|\Big)^2&\le& \int_\Om\mu^2\Big|\nabla^2f-\frac{\Id}{\mu^2}\Big|^2\int_\Om\frac{|\nabla^2 f|}{|\Id|}
    \\
    &
    \le& C(n)\,|\Om|^{3/2}\,\eta(\Om) \Big(\int_\Om|\nabla^2 f|^2\Big)^{1/2}\le C(n)\,|\Om|^2\,\eta(\Om)\,,
  \end{eqnarray*}
  that is
  \begin{equation}
    \label{stima hessiano 2}
      \int_\Om\Big|\nabla^2f-\frac{\Id}{\mu^2}\Big|\le C(n)\,|\Om|\,\sqrt{\eta(\Om)}\,.
  \end{equation}
  In particular, $|\tr(M_1)-\tr(M_2)|\le|M_1-M_2|$ and $\Delta f=1$ give us
  \[
  \int_\Om\Big|1-\frac{n+1}{\mu^2}\Big|\le C(n)\,|\Om|\,\sqrt{\eta(\Om)}\,,
  \]
  which leads to
  \[
  \int_\Om\Big|\frac{\Id}{n+1}-\frac{\Id}{\mu^2}\Big|=\sqrt{n+1}\,\int_\Om\Big|\frac{1}{n+1}-\frac1{\mu^2}\Big|
  \le C(n)\,|\Om|\,\sqrt{\eta(\Om)}\,.
  \]
  We prove \eqref{stima L1 hessiano f} by combining this last inequality with \eqref{stima hessiano 2}. We now prove \eqref{stima L2 derivata normale}. By \eqref{ros identity} one has
  \begin{eqnarray}\label{stima traccia 1}
    \frac{|\Om|}{n+1}\,\Big(\int_{\pa\Om}\frac{n}H\,-(n+1)|\Om|\Big)
    \ge2\,\int_{\pa\Om}|\nabla f|\,\Big(\Big(\int_{\pa\Om}\frac{1}{H}\,\int_{\pa\Om}|\nabla f|^2\,H\,\Big)^{1/2}-\int_{\pa\Om}|\nabla f|\,\Big)\,.
  \end{eqnarray}
  Since $\int_{\pa\Om}|\nabla f|=|\Om|$, if $\l>0$ is such that
  \[
    \l^4=\Big(\int_{\pa\Om}\frac{1}{H}\Big)^{-1}\,\int_{\pa\Om}|\nabla f|^2\,H\,,
  \]
  then \eqref{stima traccia 1} gives us
  \begin{eqnarray*}
    \frac{1}{2(n+1)}\,\Big(\int_{\pa\Om}\frac{n}H\,-(n+1)|\Om|\Big)
    &\ge&\Big(\int_{\pa\Om}\frac{1}{H}\,\int_{\pa\Om}|\nabla f|^2\,H\,\Big)^{1/2}-\int_{\pa\Om}|\nabla f|\,
    \\
    &=&\int_{\pa\Om} \frac{\l^2}2\,\frac{1}{H}+\frac1{2\l^2}|\nabla f|^2\,H-|\nabla f|
    \\
    &=&\int_{\pa\Om} \frac12\,\Big(\frac{\l}{\sqrt{H}}-\frac{|\nabla f|\,\sqrt{H}}{\l}\Big)^2
    \ge\int_{\pa\Om} \frac{H}{2\l^2}\,\Big(\frac{\l^2}{H}-|\nabla f|\Big)^2\,.
    \end{eqnarray*}
    Again, by \eqref{Omega H fra nmezzi e 2n},
    \begin{eqnarray*}
    \frac{H_0}{\l^2}\int_{\pa\Om} \Big(\frac{\l^2}{H}-|\nabla f|\Big)^2
    \le C(n)\,\eta(\Om)\,\int_{\pa\Om}\frac{n}H\le C(n)\, P(\Om)\,\frac{n}{H_0}\,\eta(\Om)\,.
    \end{eqnarray*}
    Finally, by \eqref{reilly identity f constant} one has $\int_{\pa\Om} H|\nabla f|^2\le|\Om|$, so that
   \begin{equation}
     \label{lambda alla quarta}
        \l^4\le C\,H_0\,\frac{|\Om|}{ P(\Om)}\le C\,,
   \end{equation}
   and thus
   \begin{eqnarray}\label{stima traccia 2}
   \int_{\pa\Om} \Big(\frac{\l^2}{H}-|\nabla f|\Big)^2\le C(n)\,\Big(\frac{n}{H_0}\Big)^2\, P(\Om)\,\eta(\Om)\,.
   \end{eqnarray}
   Now let $|\nabla f|_{\pa\Om}$ denote the average of $|\nabla f|$ on $\pa\Om$, so that $|\nabla f|_{\pa\Om}=|\Om|/ P(\Om)$. By \eqref{lambda alla quarta} and \eqref{stima traccia 2} we thus find
   \begin{eqnarray*}
   \int_{\pa\Om} \Big(|\nabla f|_{\pa\Om}-|\nabla f|\Big)^2&\le&\int_{\pa\Om} \Big(\frac{\l^2}{H_0}-|\nabla f|\Big)^2
   \le 2\int_{\pa\Om} \Big(\frac{\l^2}{H_0}-\frac{\l^2}{H}\Big)^2+2\int_{\pa\Om} \Big(\frac{\l^2}{H}-|\nabla f|\Big)^2
   \\
   &\le&C(n)\,\Big(\frac{n}{H_0}\Big)^2\, P(\Om)\,\de(\Om)\,,
   \end{eqnarray*}
   where in the last inequality we have used \eqref{eta Omega minore delta Omega} and $\de(\Om)\le 1$. We deduce \eqref{stima L2 derivata normale} by noticing that $|\nabla f|_{\pa\Om}=n/(n+1)H_0$.
\end{proof}

We now exploit a compactness argument to show that if the Alexandrov's deficit of $\Om$ is small enough, then $\Om$ can be taken arbitrarily close (in various ways) to a finite family of tangent balls of unit radii.

\begin{theorem}\label{thm compattezza}
   Given $n,L\in\N$, $n\ge 2$, $L\ge1$, $a\in(0,1]$, and $\tau>0$ there exists $c(n,L,a,\tau)>0$ with the following property. If $\Om$ satisfies \eqref{Omega C2 con H positiva}, \eqref{Omega riscalato H0 uguale n}, \eqref{Omega perimetro minore L+1}, $f$ is defined as in \eqref{f definizione} (and then extended to $0$ on $\R^{n+1}\setminus\Om$) and $\de(\Om)\le c(n,L,a,\tau)$, then there exists a finite family of disjoint unit balls $\{B_{z_j,1}\}_{j\in J}$ with $\#\,J\le L$ such that, setting
   \[
   G=\bigcup_{j\in J}B_{z_j,1}\,,
   \]
   $\pa G$ is connected (that is, each sphere $\pa B_{z_j,1}$ intersects tangentially at least another sphere $\pa B_{z_\ell,1}$ for some $\ell\ne j$) and
   \[
   |\Om\Delta G|+\hd(\pa\Om,\pa G)+|P(\Om)-P(G)|+\|f-f_G\|_{C^0(\R^{n+1})}\le\tau\,,
   \]
  where
  \[
  f_G(x)=-\sum_{j\in J}\max\Big\{\frac{1-|x-x_j|^2}{2(n+1)},0\Big\}\,,\qquad x\in\R^{n+1}\,.
  \]
  Moreover, there exist $\S\subset\pa G$ and $\phi\in C^{1,\g}(\S)$ for every $\g\in(0,1)$ such that $\pa G\setminus\S$ consists of at most $C(n)\,L$-many spherical caps whose diameters are bounded by $\tau$, and such that $(\Id+\phi\,\nu_G)(\S)\subset\pa\Om$ with
  \begin{eqnarray*}
  \|\phi\|_{C^1(\S)}+\H^n\big(\pa\Om\setminus(\Id+\phi\,\nu_G)(\S)\big)\le \tau\,,\qquad \|\phi\|_{C^{1,\g}(\S)}\le C(n,\g)\,.
  \end{eqnarray*}
\end{theorem}

\begin{remark}\label{remark costanti esplicite}
  {\rm Notice that by Theorem \ref{thm compattezza} and since $\|\nabla f\|_{C^0(\Om)}\le\sqrt2\,\|f\|_{C^0(\Om)}^{1/2}$ thanks to \eqref{f Lip}, one can deduce that
  \begin{equation}
    \label{f stima Cn}
    \|f\|_{C^1(\Om)}\le C_0(n)\,,
  \end{equation}
  whenever $\de(\Om)\le c(n,L,a)$. (Indeed, it is enough to pick $\tau=\tau(n)$ and use $\|f-f_G\|_{C^0(\Om)}\le\tau$.) As a consequence one can choose, in the proof of Theorem \ref{thm main 1}, if working with \eqref{f C0}--\eqref{f Lip} or with \eqref{f stima Cn}. In the former case, one obtains larger powers of $L$ but explicitly computable constants $C(n)$ in the quantitative estimates of Theorem \ref{thm main 1}; in the latter case, we obtain smaller powers of $L$ but lose the ability of computing the corresponding constants $C(n)$. We shall opt for the second possibility.}
\end{remark}

\begin{proof}[Proof of Theorem \ref{thm compattezza}]
  Let us consider a sequence of sets $\{\Om_h\}_{h\in\N}$ satisfying \eqref{Omega C2 con H positiva}, \eqref{Omega riscalato H0 uguale n} and \eqref{Omega perimetro minore L+1} (with the same $L$ and $a$ for every $h\in\N$), and correspondingly define $f_h$ starting from $\Om_h$ by \eqref{f definizione}. Assuming that $\de(\Om_h)\to 0$, it will suffice to prove that, up extracting subsequences,
  \begin{equation}
    \label{tesi compattezza1}
      \lim_{h\to\infty}|\Om_h\Delta G|+\hd(\pa\Om_h,\pa G)+|P(\Om_h)-P(G)|+\|f_h-f_G\|_{C^0(\R^{n+1})}=0\,,
  \end{equation}
  where $G$ and $f_G$ are associated to a family of balls $\{B_{z_j,1}\}_{j\in J}$ as in the statement, and that there exist $\S_h\subset\pa G$ and $\phi_h\in C^{1,\g}(\S_h)$ for every $\g\in(0,1)$ such that $\pa G\setminus\S_h$ consists of at most $C(n)\,L$-many spherical caps with vanishing diameters, and $(\Id+\phi_h\,\nu_G)(\S_h)\subset\pa\Om_h$ with
  \[
  \lim_{h\to\infty}\|\phi_h\|_{C^1(\S_h)}+\H^n\big(\pa\Om_h\setminus(\Id+\phi_h\,\nu_G)(\S_h)\big)=0\,,\qquad\sup_{h\in\N}\|\phi_h\|_{C^{1,\g}(\S_h)}\le C(n,\g)\,.
  \]
  To this end, we first note that, by \eqref{Omega diametro minore L}, up to translating the sets $\Om_h$ one has
  \begin{equation}
    \label{compattezza limitati}
  \Om_h\subset B_R\,,\qquad\forall h\in\N\,,
  \end{equation}
  where $R=R(n,L)$. By \eqref{compattezza limitati} and since $P(\Om_h)\le C(n,L)$ thanks to \eqref{Omega perimetro minore L+1}, the compactness theorem for sets of finite perimeter \cite[Theorem 12.26]{maggiBOOK} implies that, up to extracting subsequences,
  \begin{equation}
    \label{Omegah va ad Omega}
    \lim_{h\to\infty}|\Om_h\Delta  G|=0\,,
  \end{equation}
  where $G\subset B_R$ is a set of finite perimeter in $\R^{n+1}$. Similarly, if we define $\bar u_h:\R^{n+1}\to\R$ by setting $\bar u_h=-f_h$ on $\Om_h$, and $\bar u_h=0$ on $\R^{n+1}\setminus\Om_h$, then by \eqref{f Lip} and by \eqref{compattezza limitati} we find that, again up to extracting subsequences,
  \begin{equation}
    \label{compattezza uh ad u}
      \lim_{h\to\infty}\|\bar u_h-\bar u\|_{C^0(\R^{n+1})}+\|\bar u_h-\bar u\|_{L^1(\R^{n+1})}=0\,,
  \end{equation}
  where $\bar u:\R^{n+1}\to[0,\infty)$ is a Lipschitz function on $\R^{n+1}$. Now, by Remark \ref{remark D2u meno mu},
  \begin{eqnarray}
  \label{second}
  D^2\bar u_h=-\nabla^2f_h\,\Le^{n+1}\llcorner\Om_h+|\nabla f_h|\,\nu_{\Om_h}\otimes\nu_{\Om_h}\,\H^n\llcorner\pa\Om_h\,.
  \end{eqnarray}
  In particular,
  \begin{eqnarray*}
  |D^2\bar u_h|(\R^{n+1})&=&\int_{\Om_h}|\nabla^2f_h|+\int_{\pa\Om_h}|\nabla f_h|
  \\
  &\le&|\Om_h|^{1/2}\|\nabla^2 f_h\|_{L^2(\Om_h)}
  +\Big(\int_{\pa\Om_h}\frac1{H}\Big)^{1/2}\Big(\int_{\pa\Om_h}H\,|\nabla f_h|^2\Big)^{1/2}
  \\
  &\le&|\Om_h|+C(n)\,P(\Om_h)^{1/2}\,|\Om_h|^{1/2}\le C(n,L)\,,
  \end{eqnarray*}
  where in the last line we have used, in the order, \eqref{f L2 estimates}, \eqref{Omega H fra nmezzi e 2n}, \eqref{reilly identity f constant}, \eqref{Omega volume minore L+1} and \eqref{Omega perimetro minore L+1}; as a consequence,
  \begin{equation}
    \label{hessiani 1}
  D\bar u\in BV(\R^{n+1};\R^{n+1})\,,\qquad D^2\bar u_h\weak D^2\bar u\quad\mbox{as Radon measures on $\R^{n+1}$}\,.
  \end{equation}
  If $\vphi\in C^0_c(\R^{n+1})$, then by \eqref{stima L1 hessiano f} and \eqref{Omegah va ad Omega}
  \[
  (D^2\bar u_h\llcorner\Om_h)(\vphi)=-\int_{\Om_h}\,\vphi\,\nabla^2f_h\to-\frac{\Id}{n+1}\,\int_G\,\vphi\,,
  \]
  so that
  \begin{equation}
    \label{hessiani notte}
        D^2\bar u_h\llcorner\Om_h\weak -\frac{\Id}{n+1}\,\Le^{n+1}\llcorner G\quad \mbox{as Radon measures in $\R^{n+1}$}\,.
  \end{equation}
  By \eqref{hessiani 1} and \eqref{hessiani notte}, if $\mu$ denotes the weak-* limit of the Radon measures
  \[
  \mu_h=D^2\bar u_h\llcorner(\R^{n+1}\setminus\Om_h)=|\nabla f_h|\,\nu_{\Om_h}\otimes\nu_{\Om_h}\,\H^n\llcorner\pa\Om_h\,,
  \]
  then we have
  \begin{equation}
    \label{hessiani 2}
      D^2\bar u=-\frac{\Id}{n+1}\,\Le^{n+1}\llcorner G+\mu\,.
  \end{equation}
  We claim that
  \begin{equation}
    \label{inclusione uno}
    |\{\bar u>0\}\setminus G|=0\,,\qquad \spt\mu\cap\{\bar u>0\}=\emptyset\,.
  \end{equation}
  To prove the first part of \eqref{inclusione uno}, we note that if $\bar u(x)>0$, then by uniform convergence $\bar u_h\ge\bar u(x)/2$ on $B_{x,s_x}$ for every $h\ge h_x$ and for some $s_x>0$, so that $B_{x,s_x}\subset\Om_h$ for every $h\ge h_x$. This implies that $|B_{x,s_x}\setminus G|=0$ (thus the first part of \eqref{inclusione uno}), and also that $B_{x,s_x}\cap\spt\mu_h=\emptyset$: since $\spt\mu$ is contained in the set of the accumulation points of sequences $\{x_h\}_{h\in\N}$ with $x_h\in\pa\Om_h$, we have proved \eqref{inclusione uno}. By combining \eqref{hessiani 2} and \eqref{inclusione uno} we deduce that
  \begin{equation}
    \label{solito}
      D^2\bar u\llcorner\{\bar u>0\}=-\frac{\Id}{n+1}\,\Le^{n+1}\llcorner\{\bar u>0\}\,.
  \end{equation}
  Now let $\{A_j\}_{j\in J}$ denote the connected components of the open set $\{\bar u>0\}$, then by \eqref{solito}
  we can find $z_j\in\R^{n+1}$ and $c_j\in\R$ such that
  \[
  \bar u(x)=c_j-\frac{|x-z_j|^2}{2(n+1)}\,,\qquad\forall\,x\in A_j\,,
  \]
  and since $\bar u\ge0$ it must be
  \[
  c_j\ge0\,,\qquad A_j\subset B_{z_j,s_j}\qquad\mbox{where $s_j=(2(n+1)c_j)^{1/2}$}\,,
  \]
  thus $c_j>0$ because $A_j$ is open. In conclusion,
  \[
  \{\bar u>0\}=\bigcup_{j\in J}A_j\subset\bigcup_{j\in j}B_{z_j,s_j}\subset\{\bar u>0\}\,,
  \]
  that is, $\bar u=-f_G$,
  \begin{equation}
    \label{struttura bar u}
  \bar u(x)=\sum_{j\in J}\max\Big\{\frac{s_j^2-|x-z_j|^2}{2(n+1)},0\Big\}\,,\qquad A_j=B_{z_j,s_j}\,.
  \end{equation}
  We now want to prove that $|G\Delta\{\bar u>0\}|=0$ and that $J$ is finite with $s_j=1$ for every $j\in J$. To this end we first notice that $s_j\le 1$ for every $j\in J$. Indeed, by \eqref{struttura bar u} we have that
  \[
  \{\bar u>\e\}=\bigcup_{j\in J}B_{z_j,\sqrt{(s_j^2-2(n+1)\e)_+}}\,,\qquad\forall \e>0\,,
  \]
  so that, by uniform convergence,
  \[
  \bigcup_{j\in J}B_{z_j,\sqrt{(s_j^2-2(n+1)\e)_+}}\subset\big\{u_h>\frac\e2\big\}\subset\Om_h\,,\qquad\forall h\ge h_\e\,.
  \]
  In particular, if we fix $j\in J$, pick $\e<s_j^2/2(n+1)$, and let $h\ge h_{\e,j}$, then by the previous inclusion there exists $y\in\pa\Om_h$ such that
  \[
  \frac{n}{\sqrt{s_j^2-2(n+1)\e}}\ge H_{\pa\Om_h}(y)\ge n(1-\de(\Om_h))\,,
  \]
  that is, letting $h\to\infty$, $s_j^2-2(n+1)\e\le 1$. By the arbitrariness of $\e$, we conclude that $s_j\le 1$. We now apply Pohozaev's identity \eqref{pohozaev} to $f_h$ to find
  \[
  (n+3)\int_{\R^{n+1}}\,\bar u_h=(n+3)\int_{\Om_h}\,(-f_h)=\int_{\pa\Om_h}(x\cdot\nu_{\Om_h})|\nabla f_h|^2\,,
  \]
  so that by \eqref{compattezza uh ad u}, \eqref{stima L2 derivata normale}, and the divergence theorem we find
  \[
  (n+3)\int_{\R^{n+1}}\bar u=\lim_{h\to\infty}\int_{\pa\Om_h}\frac{(x\cdot\nu_{\Om_h})}{(n+1)^2}=\frac{|G|}{n+1}\ge\frac{|B|}{n+1}\,\sum_{j\in J}s_j^{n+1}\,.
  \]
  At the same time, by \eqref{struttura bar u} and a simple computation we find
  \[
  (n+1)\,(n+3)\int_{\R^{n+1}}\bar u=|B|\,\sum_{j\in J}s_j^{n+3}\,,
  \]
  so that
  \[
  \sum_{j\in J}s_j^{n+1}(1-s_j^2)\le0\,.
  \]
  Since $s_j\in(0,1]$ for every $j\in J$, we conclude that $s_j=1$ for every $j\in J$. As a consequence, $\#\,J\le L$, because of
  \[
  (L+1-a)|B|\ge \lim_{h\to\infty}|\Om_h|=|G|\ge|\{\bar u>0\}|=\#J\,|B|\,.
  \]
  Since $J$ is finite we deduce from \eqref{struttura bar u} that
  \[
  D^2\bar u=-\frac{\Id}{n+1}\,\Le^{n+1}\llcorner\bigcup_{j\in J}B_{z_j,1}+\sum_{j\in J}\frac{\nu_{B_{z_j,1}}\otimes \nu_{B_{z_j,1}}}{n+1}\,\H^n\llcorner\pa B_{z_j,1}\,.
  \]
  By comparing this formula with \eqref{hessiani 2} we conclude that $|G\Delta\{\bar u>0\}|=0$, provided we can show that the measure $\mu$ appearing in \eqref{hessiani 2} is singular with respect to $\Le^{n+1}$, of course. To this end, it suffices to consider the multiplicity one varifolds $V_h$ associated to $\pa\Om_h$. Since (in the notation and terminology of \cite[Chapter 8]{SimonLN}) the varifolds $\{V_h\}_{h\in\N}$ have uniformly bounded masses (as $\mathbf{M}(V_h)=\H^n(\pa\Om_h)$) and uniformly bounded generalized mean curvatures (thanks to \eqref{Omega H fra nmezzi e 2n}), by \cite[Theorem 42.7, Remark 42.8]{SimonLN} there exists an integer multiplicity rectifiable $n$-varifold $V$ such that $V_h\weak V$ as varifolds. In particular, if $V$ is supported on the $n$-rectifiable set $M$, and if $\theta$ denotes the integer multiplicity of $V$, then, denoting by $\nu_M$ a Borel vector-field such that $\nu_M(x)^\perp=T_xM$ for $\H^n$-a.e. $x\in M$, we get
  \[
  \int_{M}\vphi\,\theta\,\nu_M\otimes\nu_M\,d\H^n=\lim_{h\to\infty}\int_{\pa\Om_h}\vphi\,\nu_{\Om_h}\otimes\nu_{\Om_h}\,d\H^n\,,\qquad\forall\vphi\in C^0_c(\R^{n+1})\,.
  \]
  Hence, by \eqref{stima L2 derivata normale} and by definition of $\mu_h$ and $\mu$ we conclude that
  \[
  \mu=\frac{\theta}{n+1}\,\nu_M\otimes\nu_M\,\H^n\llcorner M\,.
  \]
  As explained this shows that $|G\Delta\{\bar u>0\}|=0$, and thus, from now one we directly set
  \[
  G=\bigcup_{j\in J}B_{z_j,1}\,.
  \]
  Let us prove that $P(\Om_h)\to P(G)$. By the divergence theorem,
  \[
  \big|(n+1)|\Om_h|-P(\Om_h)\big|=\Big|\int_{\pa\Om_h}\Big(1-\frac{H_{\pa\Om_h}}n\Big)\,(x\cdot\nu_{\Om_h})\Big|\le \diam(\Om_h)\,\de(\Om_h)\,,
  \]
  while at the same time $(n+1)|G|=P(G)$, so that
  \begin{equation}
    \label{compattezza POmh meno PG stima}
      |P(\Om_h)-P(G)|\le (n+1)||\Om_h|-|G||+\diam(\Om_h)\,\de(\Om_h)\le (n+1)|\Om_h\Delta G|+\diam(\Om_h)\,\de(\Om_h)\,,
  \end{equation}
  and $P(\Om_h)\to P(G)$, as claimed. This last fact implies in particular that
  \begin{equation}
    \label{convergenza totale}
      \H^n\llcorner\pa\Om_h\weak\H^n\llcorner\pa G\quad \mbox{as Radon measures in $\R^{n+1}$}\,.
  \end{equation}
  By \eqref{convergenza totale}, \eqref{Omega stima densita perimetro basso} and a classical argument we immediately prove that $\hd(\pa\Om_h,\pa G)\to 0$.
  Since $\pa\Om_h$ is connected for every $h$, $\hd(\pa\Om_h,\pa G)\to 0$ implies that $\pa G$ is connected. We are thus left to prove the existence of sets $\S_h$ and maps $\phi_h$ with the claimed properties. To this end we put the proof of the theorem on hold, and recall some basic useful facts from the regularity theory for integer rectifiable varifolds.
  \end{proof}

Given $x\in\R^{n+1}$, $\nu\in S^n$ and $r>0$ we set
\begin{eqnarray*}
\C_{x,r}^\nu&=&\big\{y\in\R^{n+1}:|\p_\nu(y-x)|<r\,,|(y-x)\cdot\nu|<r\big\}\,,\qquad\C_r=\C_{0,r}^{e_n}\,,\qquad\C=\C_1\,,
\\
\D_{x,r}^\nu&=&\big\{y\in\R^{n+1}:|\p_\nu(y-x)|<r\,,(y-x)\cdot\nu=0\big\}\,,\qquad\D_r=\D_{0,r}^{e_n}\,,\qquad\D=\D_1\,,
\end{eqnarray*}
where $\p_\nu(v)=v-(v\cdot\nu)\nu$ for every $v\in\R^{n+1}$. Given $u\in C^{k,\g}(\D_r)$, it will be useful to consider, along with the standard $C^{k,\g}$-norms on $\D_r$, the scaled norms
\[
\|u\|_{C^{k,\g}(\D_r)}^*=\sum_{j=0}^k\,r^{j-1}\,\|D^j u\|_{C^0(\D_r)}+r^{k-1+\g}\,[D^ku]_{C^{0,\g}(\D_r)}\,,
\]
which are invariant by scaling in the sense that, if we set $\l_r(u)(x)=r^{-1}\,u(r\,x)$ for $x\in\D$, then
\[
\|\l_r(u)\|_{C^{k,\g}(\D)}=\|\l_r(u)\|_{C^{k,\g}(\D)}^*=\|u\|_{C^{k,\g}(\D_r)}^*\,,\qquad\forall r>0\,.
\]
We shall need the following technical lemma, which just amounts to a simple application of the implicit function theorem, and whose proof can be found in \cite[Lemma 4.3]{CiLeMaIC1}. In the statement, given $u:\D_{4r}\to\R$ with $|u|< 4r$ on $\D_{4r}$, we set
\[
\Gamma_r(u)=(\Id+u\,e_n)(\D_{4r})\subset\C_{4r}\,.
\]

\begin{lemma}\label{lemma facile}  Given $n\ge1$, $M>0$ and $\g\in[0,1]$ there exist positive constants $\k_0=\k_0(n,M,\g)<1$ and $\k_1=\k_1(n,M,\g)$ with the following property. If $u_1\in C^{2,1}(\D_{4r})$, $u_2\in C^{1,\g}(\D_{4r})$, and
\[
\max_{i=1,2}\|u_i\|_{C^1(\D_{4r})}^*\le \k_0\,,\qquad
\max\big\{\|u_1\|_{C^{2,1}(\D_{4r})}^*,\|u_2\|_{C^{1,\g}(\D_{4r})}^*\big\}\le M\,,
\]
then there exists $\psi\in C^{1,\g}(\C_{2r}\cap\Gamma_r(u_1))$ such that
\begin{gather*}
\C_r\cap\Gamma_r(u_2)\subset (\Id+\psi\nu)(\C_{2r}\cap\Gamma_r(u_1))\subset \Gamma_r(u_2)\,,
\\
\frac{\|\psi\|_{C^0(\C_{2r}\cap\Gamma_r(u_1))}}r+\|\nabla\psi\|_{C^0(\C_{2r}\cap\Gamma_r(u_1))}
+r^\g\,[\nabla\psi]_{C^{0,\g}(\C_{2r}\cap\Gamma_r(u_1))}\le \k_1\,,
\\
\frac{\|\psi\|_{C^0(\C_{2r}\cap\Gamma_r(u_1))}}{r}+\|\nabla\psi\|_{C^0(\C_{2r}\cap\Gamma_r(u_1))}\le \k_1\,\|u_1-u_2\|_{C^1(\D_{4r})}\,.
\end{gather*}
Here, $\nu\in C^{1,1}(\Gamma_r(u_1);S^n)$ is the normal unit vector field to $\Gamma_r(u_1)$ defined by
\[
  \nu(z,u_1(z))=\frac{(-\nabla u_1(z),1)}{\sqrt{1+|\nabla u_1(z)|^2}}\,,\qquad \forall z\in\D_{4r}\,.
\]
\end{lemma}

Next, let $S$ be a $\H^n$-rectifiable set in $\R^{n+1}$ with bounded generalized mean curvature in some open set $V$, that is, there exists ${\bf H}\in L^\infty(V;\H^n\llcorner S)$ such that
\[
\int_S\,\Div^S\,X\,d\H^n=\int_S\,X\cdot{\bf H}\,d\H^n\,,\qquad\forall  X\in C^1_c(V;\R^{n+1})\,,
\]
and assume that $S=\spt(\H^n\llcorner S)$, i.e., $\H^n(S\cap B_{x,r})>0$ for every $x\in S$, $r>0$. Set
\begin{equation}
  \label{allard deficit}
  \s(S,x,r)=r\,\|{\bf H}\|_{L^\infty(B_{x,r};\H^n\llcorner S)}+\max\Big\{\frac{\H^n(S\cap B_{x,r})}{\om_n\,r^n}-1,0\Big\}\,,\qquad x\in S\,,r>0\,,
\end{equation}
where $\om_n=\H^n(B\cap\{x_1=0\})$. Then for every $\g\in(0,1)$, Allard's regularity theorem \cite{Allard} (as stated in \cite[Theorem 24.2]{SimonLN} -- see also \cite[Theorem 3.2]{DeLellisNOTESallard}) gives us positive constants $\s_0(n,\g)<1$ and $C(n,\g)$ with the following property:

\medskip

\noindent {\bf Allard's theorem}: {\it With $S$ as above, if $x\in S$ and $r>0$ are such that $B_{x,r}\cc V$ and
\begin{equation}
  \label{allard hp}
  \s(S,x,r)\le \s_0(n,\g)\,,
\end{equation}
then there exist $\nu\in S^n$ and a Lipschitz map $u:(x+\nu^\perp)\to\R$ with $u(x)=0$ such that
\[
S\cap\C_{x,\s_0\,r}^\nu=\big\{z+u(z)\nu:z\in\D_{x,\s_0\,r}^\nu\big\}\,,\qquad \|u\|_{C^{1,\g}(\D_{x,\s_0\,r}^\nu)}^*\le C(n,\g)\,\s(S,x,r)^{1/4n}\,.
\]}
(Note that this statement is a particular case of Allard's theorem in the sense that we consider only density one varifolds and we restrict to the codimension one case.) In the following we shall apply this theorem with $\g=1/4n$. Correspondingly, we simply set
\[
\s_0(n)=\s_0\Big(n,\frac1{4n}\Big)<1\,.
\]
We now prove a technical lemma
\begin{figure}
  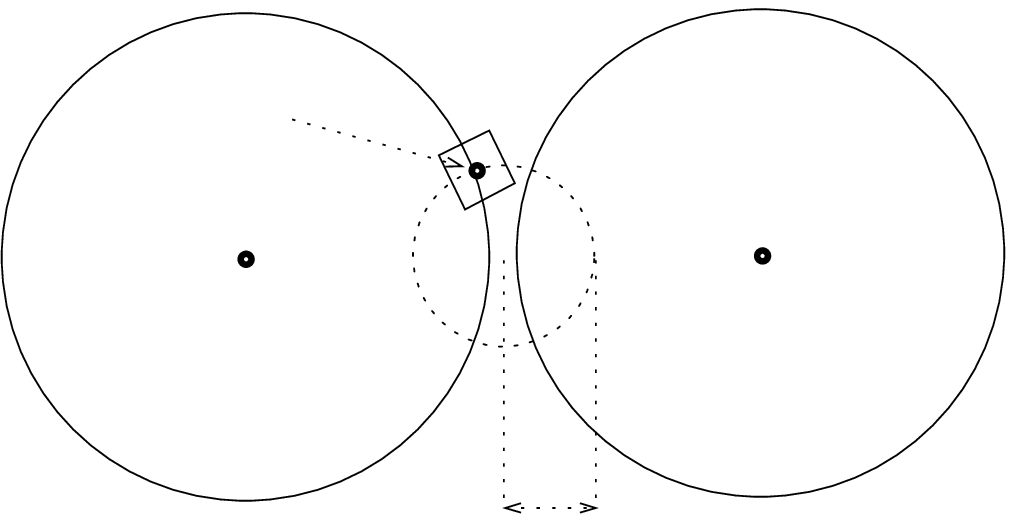\caption{{\small If $x\in\Sigma_\l\cap\pa B_{z_j,1}$, then $\pa G\cap \C_{x,c_0(n)\,\l^2}^{\mu_x}=\pa B_{z_j,1}\cap \C_{x,c_0(n)\,\l^2}^{\mu_x}$, see also \eqref{bordo G locale x}. Here $\mu_x=\nu_G(x)=\nu_{B_{z_1,1}}(x)$.}}\label{fig alex_sigmalambda}
\end{figure}
which will be useful in the proof of Theorem \ref{thm main 1} too.

\begin{lemma}
  \label{lemma allard}
  There exist positive constants $\l(n)<1$ and $c_0(n)$ with the following property. Let $\Om$ satisfy \eqref{Omega C2 con H positiva}, \eqref{Omega riscalato H0 uguale n}, and \eqref{Omega perimetro minore L+1}, let $\{B_{z_j,1}\}_{j\in J}$ be a disjoint family of unit balls, and set
  \begin{equation}
    \label{definizione Sigma lambda}
      G=\bigcup_{j\in J}B_{z_j,1}\,,\qquad\S_\l=\pa G\setminus \bigcup_{j, \ell\in J,j\ne\ell}B_{(z_j+z_\ell)/2,\l}\qquad\l>0\,.
  \end{equation}
  Assume that to each $\l\le\l(n)$ and $x\in\S_\l$ one can associate $\rho_x\in(0,1)$ and $y\in\pa\Om$ in such a way that
  \begin{gather}\label{lallard taglia rhox}
    \frac{c_0(n)\,\l^2}{2}\le \rho_x\le c_0(n)\,\l^2\,,
    \\\label{lallard taglia x-y}
    |x-y|=\dist(x,\pa \Om)\le\frac{\s_0(n)\rho_x^2}2\,,
    \\\label{lallard sigma piccolo}
    \s(\pa\Om,y,\rho_x)\le \s_0(n)\,\l(n)\,,
    \\
    \label{misteri della fede}
    |\Om\Delta G|\le C(n)\,\rho_x^{n+1}\,\sigma(\pa\Om,y,\rho_x)^{1/4n} \,.
  \end{gather}
  Then for every $\l\le\l(n)$ there exists $\psi^\l:\S_\l\to\R$ such that
  \begin{gather}\label{psilambda nomra c1gamma}
  \|\psi^\l\|_{C^{1,\g}(\S_\l)}\le C(n,\g)\,,\qquad\forall \g\in(0,1)\,,
  \\\label{psilambda C0C1}
  \l^{-2}\,\|\psi^\l\|_{C^0(\S_\l)}+\|\nabla\psi^\l\|_{C^1(\S_\l)}\le C(n)\,\max_{x\in\S_\l}\s(\pa\Om,y,\rho_x)^{1/4n}\,,
  \\\label{psilambda Nthetaintorno}
  (\Id+\psi^\l\nu_G)(\S_\l)\subset\pa\Om\,.
  \end{gather}
  \end{lemma}

\begin{remark}\label{remark geometry of Sigma lambda}
  {\rm Note that we do not assume $\pa G$ to be connected. In other words, the balls $B_{z_j,1}$ need not to be tangent, although the may be arbitrarily close or even mutually tangent, and actually this last case will somehow be the ``worst'' case to keep in mind. We also note that for $\l(n)$ small enough, if $\l\le\l(n)$, then $\pa G\setminus\S_\l$ consists of finitely many (precisely, at most $C(n)\#\,J$-many) spherical caps of diameters bounded by $C\,\l$.}
\end{remark}

\begin{proof}[Proof of Lemma \ref{lemma allard}] We {\it claim} that for every $\l\le\l(n)$ and $x\in\S_\l$ there exists
\[
\psi_x\in C^{1,1/4n}(\C_{x,\s_0\,\rho_x/4}^{\mu_x}\cap\pa G)
\]
such that
\begin{gather}\nonumber
\C_{x,\s_0\,\rho_x/8}^{\mu_x}\cap\pa\Om
\subset (\Id+\psi_x\nu_G)(\C_{x,\s_0\,\rho_x/4}^{\mu_x}\cap\pa G)\subset \C_{x,\s_0\,\rho_x/2}^{\mu_x}\cap\pa\Om\,,
\\\label{epici 3 x}
\|\psi_x\|_{C^{1,1/4n}(\C_{x,\s_0\,\rho_x/4}^{\mu_x}\cap\pa G)}\le C(n)\,,
\\\nonumber
\rho_x^{-1}\,\|\psi_x\|_{C^0(\C_{x,\s_0\,\rho_x/4}^{\mu_x}\cap\pa G)}+\|\nabla\psi_x\|_{C^0(\C_{x,\s_0\,\rho_x/4}^{\mu_x}\cap\pa G)}\le C(n)\,\s(\pa\Om,y,\rho_x)^{1/4n}\,.
\end{gather}
Postponing for the moment the proof of the claim, let us show how it allows one to complete the proof of the lemma. Indeed, \eqref{epici 3 x} implies that, for every $x_1,x_2\in\S_\l$, $\psi_{x_1}=\psi_{x_2}$  on the intersection of their respective domains of definition. Then, by setting
\[
\psi^\l(z)=\psi_x(z)\,,\qquad\forall z\in \C_{x,\s_0\,\rho_x/4}^{\mu_x}\cap\pa G\,,x\in\S_\l\,,
\]
one defines a function $\psi^\l\in C^{1,1/4n}(\S_\l)$ such that \eqref{psilambda nomra c1gamma}, \eqref{psilambda C0C1} and \eqref{psilambda Nthetaintorno} hold. Moreover, \eqref{psilambda nomra c1gamma} follows from elliptic regularity, as $\|\psi^\l\|_{C^{1,1/4n}(\S_\l)}\le C(n)$ and the mean curvature of the graph of $\psi^\l$ over $\S_\l$ is the mean curvature of $\Om$, and thus it is bounded and continuous.

We are thus left to prove our claim. To this end we consider $r(n)>0$ such that
  \begin{eqnarray}\label{rn definizione x}
  \H^n(B_{z,s}\cap \pa B)\le (1+C(n)\,s^2)\om_n\,s^n\,,\qquad \forall z\in\pa B\,,\forall s<r(n)\,,
  \\
  \label{rn uniform x}
  \sup\big\{|(p-z)\cdot\nu_B(z)|:p\in B_{z,s}\cap\pa B\big\}\le C(n)\,s^2\,,\qquad \forall z\in\pa B\,,\forall s<r(n)\,,
\end{eqnarray}
we fix $x\in\S_\l$, $\l\le\l(n)$, let $y$ and $\rho_x$ be as in the statement, and set
\[
\mu_x=\nu_G(x)=\nu_{B_{z_j,1}}(x)\qquad\mbox{for the unique $j\in J$ such that $x\in\pa B_{z_j,1}$}\,.
\]
If $c_0(n)\,\l(n)^2\le r(n)$, then by definition of $\mu_x$ there exist $\hat C(n)$ and a Lipschitz map $w_x:(x+\mu_x^\perp)\to\R$ such that
\begin{equation}
\label{bordo G locale x}
  \begin{split}
    \pa G\cap\C_{x,c_0(n)\,\l^2}^{\mu_x}=\pa B_{z_j,1}\cap\C_{x,c_0(n)\,\l^2}^{\mu_x}=\big\{z+w_x(z)\mu_x:z\in\D_{x,c_0(n)\,\l^2}^{\mu_x}\big\}\,,
\\
\|w_x\|_{C^{2,1}(\D_{x,c_0(n)\,\l^2}^{\mu_x})}\le \hat C(n)\,,\quad \|w_x\|_{C^1(\D_{x,r}^{\mu_x})}^*\le \hat C(n)\,r\,,\quad\forall r\le c_0(n)\,\l^2\,.
  \end{split}
\end{equation}
By \eqref{lallard sigma piccolo} and by Allard's theorem, there exist $\nu_x\in S^n$ and a Lipschitz map $u_x:(y+\nu_x^\perp)\to\R$ such that $u_x(y)=0$,
\begin{equation}
  \begin{split}
    \label{stima allard x}
&\pa\Om\cap\C_{y,\s_0\,\rho_x}^{\nu_x}=\big\{z+u_x(z)\nu_x:z\in\D_{y,\s_0\,\rho_x}^{\nu_x}\big\}\,,
\\
&\|u_x\|_{C^{1,1/4n}(\D_{y,\s_0\,\rho_x}^{\nu_x})}^*\le C(n)\,\s(\pa\Om,y,\rho_x)^{1/4n}\,.
  \end{split}
\end{equation}
Now we let
\begin{eqnarray*}
  K_y=\big\{z\in \C_{y,\s_0\,\rho_x}^{\nu_x}:(z-y)\cdot\nu_x\le 0\big\}\,,\quad
  K_x=\big\{z\in \C_{x,\s_0\,\rho_x}^{\mu_x}:(z-x)\cdot\mu_x\le 0\big\}\,.
\end{eqnarray*}
Up to switching $\nu_x$ with $-\nu_x$, and since $|u_x|\le C(n)\,\rho_x\,\s(\pa\Om,y,\rho_x)^{1/4n}$ on $\D_{y,\s_0\,\rho_x}^{\nu_x}$  by \eqref{stima allard x} we can assume that
\begin{equation}
  \label{stima K x}
  \big|(K_y\Delta\Om)\cap \C_{y,\s_0\,\rho_x}^{\nu_x}\big|\le C(n)\,\rho_x^{n+1}\,\s(\pa\Om,y,\rho_x)^{1/4n}\,.
\end{equation}
Similarly, by \eqref{bordo G locale x} we have $|w_x|\le C(n)\,\rho_x^2$ on $\D_{x,\s_0\,\rho_x}^{\mu_x}$, and thus
\begin{equation}
  \label{stima K xx}
  \big|(K_x\Delta G)\cap \C_{x,\s_0\,\rho_x}^{\mu_x}\big|\le C(n)\,\rho_x^{n+2}\le  C(n)\,\rho_x^{n+1}\,\s(\pa\Om,y,\rho_x)^{1/4n}\,,
\end{equation}
as \eqref{allard deficit} and \eqref{Omega H fra nmezzi e 2n} imply $\rho_x\le C(n)\s(\pa\Om,y,\rho_x)$. Since $|y-x|\le\s_0\,\rho_x/2$ by \eqref{lallard taglia x-y} and $\rho_x<1$, we find
\[
B_{x,\s_0\,\rho_x/2}\subset B_{y,\s_0\,\rho_x}\subset \C_{y,\s_0\,\rho_x}^{\nu_x}\,,\qquad\mbox{as well as $B_{x,\s_0\,\rho_x/2}\subset
\C_{x,\s_0\,\rho_x}^{\mu_x}$ of course}\,,
\]
and thus, by \eqref{stima K x} and \eqref{stima K xx},
\begin{eqnarray}\nonumber
  |\Om\Delta G|\ge |(\Om\Delta G)\cap B_{x,\s_0\,\rho_x/2}|
  &\ge&|(K_x\Delta K_y)\cap B_{x,\s_0\,\rho_x/2}|-C(n)\,\rho_x^{n+1}\,\s(\pa\Om,y,\rho_x)^{1/4n}
  \\\label{stima cilindri 0x}
  &\ge&\big|\big((K_y+x-y)\Delta K_x\big)\cap B_{x,\s_0\,\rho_x/2}\big|
  \\\nonumber
  &&-C(n)\,\rho_x^{n+1}\,\s(\pa\Om,y,\rho_x)^{1/4n}
  -\big|\big(K_y+x-y\big)\Delta K_y\big|\,.
\end{eqnarray}
On the one hand, for every $z\in\R^{n+1}$, $r>0$ and $\nu,\nu'\in S^n$ one has
\begin{equation}
  \label{stima cilindri 1x}
  \Big|\Big(\big\{p\in\C_{z,r}^\nu:(p-z)\cdot\nu\le 0\big\}\Delta
  \big\{p\in\C_{z,r}^{\nu'}:(p-z)\cdot\nu'\le 0\big\}\Big)\cap B_{z,r/2}\Big|\ge c(n)\,|\nu-\nu'|\,r^{n+1}\,;
\end{equation}
on the other hand, again by $|y-x|\le\s_0\,\rho_x^2/2$,
\begin{equation}
  \label{stima cilindri 2x}
|K_y\Delta(x-y+K_y)|\le C(n)\,P(K_y)\,|y-x|\le C(n)\,\rho_x^n\,|y-x|\le  C(n)\,\rho_x^{n+2}\,.
\end{equation}
By \eqref{stima cilindri 0x}, \eqref{stima cilindri 1x}, and \eqref{stima cilindri 2x} we conclude that
\[
c(n)\,|\nu_x-\mu_x|\,\rho_x^{n+1}\le C(n)\,\,\rho_x^{n+1}\,\sigma(\pa\Om,y,\rho_x)^{1/4n}+|\Om\Delta G|\,,
\]
so that \eqref{misteri della fede} and \eqref{lallard sigma piccolo} give us
\begin{equation}
  \label{normali vicinex}
  |\nu_x-\mu_x|\le C(n)\,\s(\pa\Om,y,\rho_x)^{1/4n}\,.
\end{equation}
By \eqref{lallard taglia x-y}, \eqref{stima allard x}, and \eqref{normali vicinex}, provided $\l(n)$ is small enough, there exist a constant $C_*(n)$ and a Lipschitz map $v_x:(x+\mu_x^\perp)\to\R$ such that
\begin{gather*}
\pa\Om\cap\C_{x,\s_0\,\rho_x/2}^{\mu_x}=\big\{z+v_x(z)\mu_x:z\in\D_{x,\s_0\,\rho_x/2}^{\mu_x}\big\}\,,
  \\
  \|v_x\|_{C^{1,1/4n}(\D_{x,\s_0\,\rho_x/2}^{\mu_x})}\le C_*(n)\,,\qquad
   \|v_x\|_{C^1(\D_{x,\s_0\,\rho_x/2}^{\mu_x})}^*\le C_*(n)\,\s(\pa\Om,y,\rho_x)^{1/4n}\,.
\end{gather*}
By this last property, by \eqref{lallard taglia rhox}, and by \eqref{bordo G locale x} we can apply Lemma \ref{lemma facile} into the cylinder $\C_{x,\s_0\,\rho_x/2}^{\mu_x}$: indeed, setting by a rigid motion $x=0$ and $\mu_x=0$, and choosing
\[
4\,r=\frac{\s_0\rho_x}2\,,\qquad u_1=w_x\,,\qquad u_2=v_x\,,\qquad \g=\frac1{4n}\,,\qquad M=\max\{\hat C(n),C_*(n)\}\,,
\]
we find that
\begin{eqnarray*}
\max_{i=1,2}\|u_i\|_{C^1(\D_{4r})}^*&=&\max\big\{\|w_x\|_{C^1(\D_{x,\s_0\,\rho_x/2}^{\mu_x})}^*,\|v_x\|_{C^1(\D_{x,\s_0\,\rho_x/2}^{\mu_x})}^*\big\}
\\
&\le&\max\big\{\hat{C}(n)\frac{\s_0\,\rho_x}2,C_*(n)\,\s(\pa\Om,y,\rho_x)^{1/4n}\big\}\le C(n)\,\l(n)^{1/4n}\le\k_0\big(n,\g,M\big)\,,
\end{eqnarray*}
provided $\l(n)$ is small enough. By Lemma \ref{lemma facile}, there exists $\psi_x\in C^{1,1/4n}(\C_{x,\s_0\,\rho_x/4}^{\mu_x}\cap\pa G)$ satisfying \eqref{epici 3 x}.
\end{proof}

\begin{proof}
  [Proof of Theorem \ref{thm compattezza}, conclusion] We now conclude the proof of Theorem \ref{thm compattezza}. Let us recall the situation we left: we have $\{\Om_h\}_{h\in\N}$ satisfying \eqref{Omega C2 con H positiva}, \eqref{Omega riscalato H0 uguale n} and \eqref{Omega perimetro minore L+1} (with the same $L$ and $a$) and
  \begin{equation}
    \label{tesi compattezza1 finale}
      \lim_{h\to\infty}|\Om_h\Delta G|+\hd(\pa\Om_h,\pa G)+|P(\Om_h)-P(G)|=0\,,
  \end{equation}
  where $G$ is the union over a finite family of disjoint unit balls $\{B_{z_j,1}\}_{j\in J}$. To complete the proof of the theorem, we need to prove the existence of $\S_h\subset\pa G$ and of $\phi_h:\S_h\to\R$ such that $\pa G\setminus\S_h$ consists of at most $C(n)\,L$-many spherical caps with vanishing diameters, $(\Id+\phi_h\,\nu_G)(\S_h)\subset\pa\Om_h$, and
  \begin{eqnarray}\label{tesi compattezza2}
  \lim_{h\to\infty}\|\phi_h\|_{C^1(\S_h)}+\H^n\big(\pa\Om_h\setminus(\Id+\phi_h\,\nu_G)(\S_h)\big)=0\,,\qquad\sup_{h\in\N}\|\phi_h\|_{C^{1,\g}(\S_h)}\le C(n,\g)\,,
  \end{eqnarray}
  for every $\g\in(0,1)$. To this end, we want to apply Lemma \ref{lemma allard} to $\Om=\Om_h$. Let us fix $\l\le\l(n)$, define $\S_\l$ as in \eqref{definizione Sigma lambda}, and for every $x\in\S_\l$ let us set
  \[
  \rho_x=c_0(n)\,\l^2\,,
  \]
  so that \eqref{lallard taglia rhox} holds trivially. Let us now fix $x\in\S_\l$, and consider $y_h\in\pa\Om_h$ such that $|x-y_h|=\dist(x,\pa \Om_h)$. By $\hd(\pa\Om_h,\pa G)\to 0$ we have
  \[
  |x-y_h|\le\hd(\pa\Om_h,\pa G)\le \frac{\s_0(n)c_0(n)^2\,\l^4}2\,,\qquad\forall h\ge h_\l\,,
  \]
  provided $h_\l\in\N$ is large enough; in particular, \eqref{lallard taglia x-y} holds for $h$ large enough. Next, we notice that by \eqref{convergenza totale} and $|y_h-x|\to 0$ we have
  \[
  \limsup_{h\to\infty}P(\Om_h;B_{y_h,\rho_x})\le P(G;B_{x,\rho_x})\,,
  \]
  so that \eqref{Omega H fra nmezzi e 2n}, the definition of $\rho_x$, and \eqref{rn definizione x} (applied with $s=c_0(n)\,\l^2\le r(n)$) give us
  \begin{eqnarray}\nonumber
  \limsup_{h\to\infty}\s(\pa\Om_h,y_h,\rho_x)&\le&2n\rho_x+\frac{P(G;B_{x,\rho_x})}{\om_n\,\rho_x^n}-1
  \le C(n)\,\Big(\l^2+\frac{P(G;B_{x,c_0(n)\,\l^2})}{\om_n\,(c_0(n)\,\l^2)^n}-1\Big)
  \\\label{first bound}
  &\le & C(n)\,\l^2\le \frac{\s_0(n)\,\l}2\le \frac{\s_0(n)\,\l(n)}2\,;
  \end{eqnarray}
  in particular, \eqref{lallard sigma piccolo} holds for $h$ large enough. Finally, \eqref{Omega H fra nmezzi e 2n} and \eqref{allard deficit} imply $\sigma(\pa\Om_h,y_h,\rho_x)\ge (n/2)\rho_x\ge c(n)\l^2$, thus up to take $h_\l$ large enough to entail $|\Om_h\Delta G|\le C(n)\,\l^{n+3}$ we find that \eqref{misteri della fede} holds. By Lemma \ref{lemma allard} we conclude that for every $\l\le\l(n)$ there exists $\{\psi^\l_h\}_{h\ge h_\l}\subset C^{1,\g}(\S_\l)$ for every $\g\in(0,1)$ such that
  \begin{equation}
  \label{psilambda C0C1 h}
  (\Id+\psi_h^\l\nu_G)(\S_\l)\subset\pa\Om_h\,,\qquad \|\psi_h^\l\|_{C^{1,\g}(\S_\l)}\le C(n,\g)\,,\qquad \|\psi_h^\l\|_{C^1(\S_\l)}\le C(n)\,\l^{1/2n}\,,
  \end{equation}
  where in proving the last bound we have also taken into account the first inequality in \eqref{first bound}. Since
  \begin{eqnarray*}
    \H^n\big(\pa\Om_h\setminus(\Id+\psi_h^\l\,\nu_G)(\S_\l)\big)&\le&P(\Om_h)-P(G)+\H^n(\pa G\setminus\S_\l)\\
    &&+\big|\H^n(\S_\l)-\H^n\big((\Id+\psi_h^\l\,\nu_G)(\S_\l)\big)\big|\,,
  \end{eqnarray*}
  by $P(\Om_h)\to P(G)$ and by $\|\psi_h^\l\|_{C^1(\S_\l)}\le C(n)\,\l^{1/2n}$ we find that
  \[
  \limsup_{h\to\infty}\H^n\big(\pa\Om_h\setminus(\Id+\psi_h^\l\,\nu_G)(\S_\l)\big)\le \H^n(\pa G\setminus\S_\l)+C(n)\,\H^n(\S_\l)\,\l^{1/2n}\,.
  \]
  We complete the proof of the theorem by first considering any $\l_h\to 0$, and then by setting $\phi_h=\psi_{k(h)}^{\l_h}$ for a properly chosen $k(h)\to\infty$.
\end{proof}

We now begin the proof of Theorem \ref{thm main 1}, that is, we consider the problem of turning the qualitative information provided in Theorem \ref{thm compattezza} into quantitative estimates in terms of $\de(\Om)$. Recall that, as in the introduction, we set
\[
\a=\frac{1}{2(n+2)}\,.
\]

\begin{proof}[Proof of Theorem \ref{thm main 1}] {\it Step one}: With $f$ as in \eqref{f definizione}, let us set
\begin{equation}
 \label{def Omega eps and feps}
 \e=|\Om|^{1/(n+1)}\,\eta(\Om)^\a\,,\qquad
  \Om_\e=\{x\in\Om:\dist(x,\pa\Om)>\e\}\,,\qquad f_\e=f\star w_\e\,,
\end{equation}
where $w_\e(x)=\e^{-{(n+1)}}\,w(x/\e)$ for $w\in C^\infty_c(B)$ with $w\ge 0$, $w(x)=w(-x)$ for every $x\in\R^{n+1}$, and $\int_{\R^{n+1}}w=1$. We claim that, if $C_0(n)$ is the constant appearing in \eqref{f stima Cn}, then
  \begin{eqnarray}\label{feps nabla limitato}
  \|\nabla f_\e\|_{C^0(\Om_\e)}&\le& C_0(n)\,,
  \\
  \label{feps f in C0}
  \|f_\e-f\|_{C^0(\Om_\e)}&\le& C_0(n)\,\e\,,
  \\
  \label{D2feps identita C0}
  \big\|\nabla^2f_\e-\frac{\Id}{n+1}\big\|_{C^0(\Om_\e)}&\le& C(n)\,\eta(\Om)^\a\,,
  \end{eqnarray}
  \vspace{-0.5cm}
  \begin{eqnarray}
  \label{D2feps positivo}
  \|\nabla^2f_\e\|_{C^0(\Om_\e)}\le C(n)\,,\qquad \nabla^2f_\e(x)\ge\frac{\Id}{2(n+1)}\,,\qquad\forall x\in\Om_\e\,.
  \end{eqnarray}
 Indeed, \eqref{feps nabla limitato} and \eqref{feps f in C0} are obvious. If $x\in\Om_\e$, then \eqref{D2feps identita C0} follows by \eqref{stima L1 hessiano f} and \eqref{alpha},
\begin{eqnarray*}
  \big|\nabla^2f_\e(x)-\frac{\Id}{n+1}\big|\le\frac{\|w\|_{C^0(\R^{n+1})}}{\e^{n+1}}\,\int_\Om\Big|\nabla^2 f-\frac{\Id}{n+1}\Big|
\le C(n)\,\eta(\Om)^{(1/2)-\a(n+1)}\,.
\end{eqnarray*}
Finally, \eqref{D2feps positivo} follows from \eqref{D2feps identita C0} and \eqref{eta Omega minore delta Omega} provided $\de(\Om)\le c(n)$ for $c(n)$ small enough.

\medskip

\noindent {\it Step two}: Next we define
\begin{eqnarray}\label{def K and rho}
\rho=C_0(n)\,\e=C_0(n)\,|\Om|^{1/(n+1)}\,\eta(\Om)^\a\,,\qquad A=\{f_\e<-3\rho\}\,.
\end{eqnarray}
We claim that
\begin{equation}
\label{inclusioni livelli}
  \{f<-4\rho\}\subset A\subset\{f<-2\rho\}\,,\qquad \{f<-\rho\}\subset\Om_\e\,,
\end{equation}
and that if $\{A_i\}_{i\in I}$ are the connected components of $A$, then each $A_i$ is convex and there exist $x_i\in A_i$ and $0<r_1^i\le r_2^i<\infty$ such that
  \begin{equation}
    \label{stima rho1 su rho2}
      B_{x_i, r_1^i}\subset A_i\subset B_{x_i, r_2^i}\,,\qquad 1\ge\frac{ r_1^i}{ r_2^i}\ge1-C(n)\,\eta(\Om)^\a\ge\frac12\,,
  \end{equation}
  \begin{equation}
    \label{luis}
       r_1^i\le 1+ C_1\,\de(\Om)\,,\qquad r_2^i\le  1+ C(n)\,\de(\Om)^\a\,,
  \end{equation}
  \begin{equation}\label{poho1.3}
  \int_{\Om\setminus\bigcup_{i\in I}B_{x_i, r_1^i}}(-f)\le C(n)\,|\Om|^{(n+2)/(n+1)}\,\eta(\Om)^\a\,.
  \end{equation}
Let us first prove that $\{f<-\rho\}\subset \Om_\e$: indeed, if $f(x)<-\rho$ but there exists $y\in\pa\Om$ with $|y-x|=\dist(x,\pa\Om)\le\e$, then the segment joining $x$ to $y$ is contained in $\Om$, and thus by \eqref{f stima Cn}
\[
-\rho>f(x)\ge f(y)-C_0(n)\,|x-y|\ge f(y)-\rho=-\rho\,,
\]
a contradiction. Similarly, our choice of $\rho$ and \eqref{feps f in C0} imply the other inclusions in \eqref{inclusioni livelli}. By \eqref{D2feps positivo}, $A=\{f_\e<-3\rho\}$ is an open set with convex connected components $\{A_i\}_{i\in I}$. Let $x_i\in A_i$ be such that $f_\e(x_i)\le f_\e(x)$ for every $x\in A_i$, and define
\[
g_i(x)=\frac{|x-x_i|^2}{2(n+1)}+f_\e(x_i)\,,\qquad x\in\R^{n+1}\,.
\]
By \eqref{D2feps identita C0},
\begin{equation}
  \label{feps gi D2 stima C0 Ai}
\big|\nabla^2(f_\e-g_i)(x)\big|\le C(n)\,\eta(\Om)^\a\,,\qquad\forall x\in \Om_\e\,,
\end{equation}
so that, by the convexity of $A_i$, $g_i(x_i)=f_\e(x_i)$, and $\nabla g_i(x_i)=\nabla f_\e(x_i)=0$,
\begin{eqnarray}\label{feps gi gradienti stima C0 in Ai}
  |\nabla f_\e(x)-\nabla g_i(x)|\le C(n)\,\eta(\Om)^\a\,|x-x_i|\,,&&\qquad\forall x\in A_i\,,
  \\\label{feps gi stima C0 in Ai}
  |f_\e(x)-g_i(x)|\le C(n)\,\eta(\Om)^\a\,|x-x_i|^2\,,&&\qquad\forall x\in A_i\,.
\end{eqnarray}
Let now $ r_1^i\le r_2^i$ be such that
\[
r_1^i=\sup\big\{r>0:B_{x_i,r}\subset A_i\big\}\,, \qquad r_2^i=\inf\big\{r>0: A_i\subset B_{x_i,r}\big\}\,.
\]
By definition there are $\nu_1,\nu_2\in S^n$ such that $x_i+ r_1^i\nu_1\,,x_i+ r_2^i\nu_2\in\pa A_i\subset\{f_\e=-3\rho\}$: hence,
\begin{eqnarray}\label{pointsss}
  0&=&f_\e(x_i+ r_2^i\nu_2)-f_\e(x_i+ r_1^i\nu_1)
  \\\nonumber
  &\ge&g_i(x_i+ r_2^i\nu_2)-g_i(x_i+ r_1^i\nu_1)-C(n)\,\eta(\Om)^\a\,(( r_1^i)^2+( r_2^i)^2)
  \\\nonumber
  &=& \frac{( r_2^i)^2-( r_1^i)^2}{2(n+1)}-C(n)\,\eta(\Om)^\a\,(( r_1^i)^2+( r_2^i)^2)\,,
\end{eqnarray}
that is, setting $t= r_2^i/ r_1^i\ge 1$,
\[
C(n)\,\eta(\Om)^\a\ge\frac{( r_2^i)^2-( r_1^i)^2}{( r_1^i)^2+( r_2^i)^2}=\frac{t^2-1}{t^2+1}\ge \frac{t-1}{t}=1-\frac{ r_1^i}{ r_2^i}\,,
\]
thus proving \eqref{stima rho1 su rho2}. The second inequality in \eqref{luis} follows from \eqref{stima rho1 su rho2} and \eqref{eta Omega minore delta Omega}, while the first one is proved by noticing that $B_{x_i, r_1^i}\subset\Om$, and thus for some $x_0\in\pa\Om$  one has
$n/r_1^i\ge H(x_0)\ge H_0(1-\de(\Om))=n(1-\de(\Om))$. Finally, thanks to \eqref{inclusioni livelli} and \eqref{Omega volume minore L+1},
\begin{equation}
  \label{poho1.2}
  \int_{\Om\setminus A}(-f)\le \int_{\{-f\le 4\rho\}}(-f)\le 4\rho|\Om|\le C(n)\,|\Om|^{(n+2)/(n+1)}\,\eta(\Om)^\a\,,
\end{equation}
while \eqref{stima rho1 su rho2}, $t^{n+1}-1\le 2^{n+1}\,(t-1)$ for $t=r_2^i/r_1^i\in[1,2]$, and \eqref{f stima Cn} give us
\begin{eqnarray*}
\int_{A\setminus\bigcup_{i\in I}B_{x_i, r_1^i}}(-f)&\le& \|f\|_{C^0(\Om)}\,|B|\,\sum_{i\in I}( r_2^i)^{n+1}-( r_1^i)^{n+1}
\\
&\le& C(n)\,\eta(\Om)^\a\,\,|B|\,\sum_{i\in I}( r_1^i)^{n+1}
\\
&\le& C(n)\,\eta(\Om)^\a\,|A|\le C(n)\,|\Om|\,\eta(\Om)^\a\,,
\end{eqnarray*}
where in the last inequality we have used $|A|\le|\Om|$. This proves \eqref{poho1.3}.

\medskip

\noindent {\it Step three}: We show the existence of $\{B_{x_j,s_j}\}_{j\in J}\subset\{B_{x_i,r_1^i}\}_{i\in I}$ such that $\{s_j\}_{j\in J}$ satisfies
\begin{equation}
     \label{main thm sj-1}
  \frac{\max_{j\in J}\,|s_j-1|}{\diam(\Om)}\le C(n)\,|\Om|\,\de(\Om)^\a\,,
\end{equation}
and, if $G^*=\bigcup_{j\in J}B_{x_j,s_j}$ (so that $G^*\subset\Om$ by construction), then
\begin{gather}
  \label{main thm Omega meno G^*}
  \frac{|\Om\setminus G^*|}{|\Om|}\le  C_1(n)\,\diam(\Om)\,|\Om|\,\de(\Om)^\a\,,
  \\
  \label{main thm perimetri stima proof}
    \frac{|P(\Om)-\#\,J\,P(B)|}{P(\Om)}\,\le C(n)\,\diam(\Om)\,|\Om|\,\de(\Om)^\a\,,
    \\
    \label{starstar}
    \#\,J\le L\,,\qquad \#\,J\le C(n)\,|\Om|\,.
\end{gather}
(Note that \eqref{main thm perimetri stima proof} implies \eqref{main thm perimetri stima} by \eqref{Omega volume minore L+1} and \eqref{Omega diametro minore L}.) Having in mind to exploit the Pohozaev's identity \eqref{pohozaev} (recall the proof of Theorem \ref{thm compattezza}), we first notice that, by the divergence theorem and by \eqref{f stima Cn},
\begin{eqnarray}\nonumber
\Big|\frac{|\Om|}{n+1}-\int_{\pa\Om}(x\cdot\nu_\Om)|\nabla f|^2\Big|
&=&\Big|\int_{\pa\Om}\,(x\cdot\nu_\Om)\,\Big(\frac1{(n+1)^2}-|\nabla f|^2\Big)\Big|
\\\nonumber
&\le&\diam(\Om)\Big(\frac1{n+1}+\|\nabla f\|_{C^0(\Om)}\Big)\,\int_{\pa\Om}\,\Big|\frac1{n+1}-|\nabla f|\Big|
\\\nonumber
&\le&C(n)\,\diam(\Om)\,\int_{\pa\Om}\,\Big|\frac1{n+1}-|\nabla f|\Big|\,.
\end{eqnarray}
By \eqref{stima L2 derivata normale} and $P(\Om)\le C(n)\,|\Om|$, we thus find
\begin{eqnarray}\nonumber
  \Big|\frac{|\Om|}{n+1}-\int_{\pa\Om}(x\cdot\nu_\Om)|\nabla f|^2\Big|
  &\le&
  C(n)\,\diam(\Om)\,\Big(P(\Om)\,\int_{\pa\Om}\,\Big|\frac1{n+1}-|\nabla f|\Big|^2\Big)^{1/2}
  \\\label{poho1}
  &\le& C(n)\,\diam(\Om)\,P(\Om)\, \de(\Om)^{1/2}\,.
\end{eqnarray}
By \eqref{feps f in C0}, \eqref{feps gi stima C0 in Ai}, and $\diam(A_i)\le 2$, one has
\[
\|f-g_i\|_{C^0(A_i)}\le C(n)\,|\Om|^{1/(n+1)}\,\de(\Om)^\a\,,
\]
thus, by \eqref{poho1.3}
\begin{equation}
  \label{poho1.4}
  \Big|\int_\Om f-\sum_{i\in I}\int_{B_{x_i, r_1^i}}g_i\Big|\le C(n)\,|\Om|^{(n+2)/(n+1)}\,\de(\Om)^\a\,.
\end{equation}
With $x_i+ r_1^i\nu_1$ as in \eqref{pointsss}, by definition of $\rho$ and by \eqref{feps gi stima C0 in Ai} we have
\begin{eqnarray*}
C\,|\Om|^{1/(n+1)}\,\de(\Om)^\a&\ge&3\rho=-f_\e(x_i+ r_1^i\nu_1)
\\
&\ge&-g_i(x_i+ r_1^i\nu_1)-C(n)\,( r_1^i)^2\,\eta(\Om)^\a
\\
&=&-\frac{( r_1^i)^2}{2(n+1)}-f_\e(x_i)-C(n)\,( r_1^i)^2\,\eta(\Om)^\a\,,
\end{eqnarray*}
that is, since we definitely have $r_1^i\le 2$,
\[
-\frac{( r_1^i)^2}{2(n+1)}-f_\e(x_i)\le C(n)\,|\Om|^{1/(n+1)}\,\de(\Om)^\a\,.
\]
By this last estimate and the definition of $g_i$,
\begin{eqnarray*}
  \int_{B_{x_i, r_1^i}}(-g_i)&=&\int_{B_{x_i, r_1^i}}-f_\e(x_i)-\frac{|x-x_i|^2}{2(n+1)}
  \\
  &\le& C(n)\,|B_{x_i, r_1^i}|\,|\Om|^{1/(n+1)}\,\de(\Om)^\a+\int_{B_{x_i, r_1^i}}\frac{( r_1^i)^2-|x-x_i|^2}{2(n+1)}
  \\
  &=&C(n)\,|B_{x_i, r_1^i}|\,|\Om|^{1/(n+1)}\,\de(\Om)^\a+\frac{|B_{x_i, r_1^i}|( r_1^i)^2}{(n+3)(n+1)}\,.
\end{eqnarray*}
By combining \eqref{poho1.4} with this last inequality we find
\begin{equation}
  \label{poho2}
  \int_\Om (-f)\le \sum_{i\in I}\frac{|B_{x_i, r_1^i}|( r_1^i)^2}{(n+3)(n+1)}+C(n)\,|\Om|^{(n+2)/(n+1)}\,\de(\Om)^\a\,.
\end{equation}
By Pohozaev's identity \eqref{pohozaev}, \eqref{poho1} and \eqref{poho2}, and by taking into account that $|\Om|^{1/(n+1)}\le C(n)\,\diam(\Om)$ and \eqref{Omega riscalato H0 uguale n}, we find
\[
\diam(\Om)\,P(\Om)\,\de(\Om)^{1/2}+|\Om|^{(n+2)/(n+1)}\,\de(\Om)^\a\le C(n)\,\diam(\Om)\,|\Om|\,\de(\Om)^\a\,,
\]
and thus
\begin{eqnarray}
\label{fve}
|\Om|\le \sum_{i\in I}|B_{x_i, r_1^i}|( r_1^i)^2+C(n)\,\diam(\Om)\,|\Om|\,\de(\Om)^\a\,.
\end{eqnarray}
By $|\Om|\ge \sum_{i\in I}|B_{x_i, r_1^i}|$ we finally get
\begin{eqnarray}\label{cookies}
&&|B|\,\sum_{i\in I}(r_1^i)^{n+1}\,\Big(1-( r_1^i)^2\Big)
\le C(n)\,\diam(\Om)\,|\Om|\,\de(\Om)^\a\,.
\end{eqnarray}
Let us set $\vphi(r)=r^{n+1}\,(1-r^2)$, $r\ge0$, and note that
\begin{equation}
  \label{vphi lb}
  \vphi(r)\ge
\left\{
\begin{split}
\frac34\,\,r^{n+1}\,,&\qquad\mbox{if $0\le r\le \frac{1}{2}$}\,,
\\
\frac{1-r}{2^{n+1}}\,,&\qquad\mbox{if $\frac{1}{2}\le r\le 1$}\,.
\end{split}
\right .
\end{equation}
With $C_1$ as in \eqref{luis}, let us now set
\begin{eqnarray*}
I^*=\big\{i\in I: 1\le r_1^i\le 1+C_1\,\de(\Om)^\a\big\}\,,
\qquad
I^{**}=\big\{i\in I: \frac{1}{2}\le r_1^i\le 1\big\}\,.
\end{eqnarray*}
Since $B_{x_i,r_1^i}\subset\Om$ for each $i\in I$ and by $\de(\Om)\le c(n)$ we find
\[
\#\,I^*\le\frac{|\Om|}{|B|} \,,\qquad 0\ge \vphi(r_1^i)\ge-C(n)\,\de(\Om)^\a\,, \qquad\forall i\in I^*\,,
\]
so that
\begin{equation}
  \label{u1}
  -|B|\sum_{i\in I^*}\vphi(r_1^i)\le C(n)\,|\Om|\,\de(\Om)^\a\,.
\end{equation}
By combining \eqref{u1} with  \eqref{cookies} one finds
\begin{equation}
  \label{cookies2}
  |B|\sum_{i\in I\setminus I^*}\vphi(r_1^i)\le C(n)\,\diam(\Om)\,|\Om|\,\de(\Om)^\a\,.
\end{equation}
Since $\vphi(r_1^i)\ge0$ for every $i\in I\setminus I^*$, \eqref{vphi lb} implies that
\[
  \frac34\,\sum_{i\in I\setminus(I^*\cup I^{**})}|B_{x_i,r_1^i}|\le |B|\sum_{i\in I\setminus(I^*\cup I^{**})}\vphi(r_1^i)\,,
\]
and thus, by \eqref{cookies2},
\begin{equation}
  \label{pizza}
  \sum_{i\in I\setminus(I^*\cup I^{**})}|B_{x_i,r_1^i}|\le C(n)\,\diam(\Om)\,|\Om|\,\de(\Om)^\a\,.
\end{equation}
We now prove that $r_1^i$ is close to $1$ for every $i\in I^{**}$. Indeed, by exploiting again the fact that $\vphi(r_1^i)\ge0$ for every $i\in I\setminus I^*$, together with \eqref{cookies2} and \eqref{vphi lb}, we find that
\[
\frac1{2^{n+1}} \sum_{i\in I^{**}}(1-r_1^i)\le
C(n)\,\diam(\Om)\,|\Om|\,\de(\Om)^\a\,,
\]
which in particular gives
\begin{equation}
  \label{stimetta}
1\ge r_1^i\ge 1-C(n)\,\diam(\Om)\,|\Om|\,\de(\Om)^\a\,,\qquad\forall i\in I^{**}\,.
\end{equation}
Finally, if we set $J=I^*\cup I^{**}$ and $s_j=r_1^j$ for $j\in J$, then \eqref{main thm sj-1} follows from \eqref{stimetta} and the definition of $I^*$, while \eqref{fve}, \eqref{pizza} and \eqref{main thm sj-1} give us
\begin{eqnarray*}
|\Om|&\le&\sum_{j\in J}|B_{x_j, s_j}|\,s_j^2+C(n)\,\diam(\Om)\,|\Om|\,\de(\Om)^\a
\\
&\le&(1+C(n)\,\diam(\Om)\,|\Om|\,\de(\Om)^\a)\,|G^*|+C(n)\,\diam(\Om)\,|\Om|\,\de(\Om)^\a\,,
\end{eqnarray*}
i.e.
\[
|\Om\setminus G^*|\le C(n)\,\diam(\Om)\,|\Om|\,|G^*|\,\de(\Om)^\a\le C(n)\,\diam(\Om)\,|\Om|^2\,\de(\Om)^\a\,.
\]
This proves \eqref{main thm Omega meno G^*}. Now by \eqref{main thm sj-1} and since $s_j\ge 1/2$
\begin{eqnarray*}
|P(G^*)-(n+1)|G^*||&=&(n+1)|B|\sum_{j\in J}s_j^n\,|s_j-1|\le C(n)\,\max_{j\in J}|s_j-1|\,|B|\sum_{j\in J}s_j^{n+1}\,
\\
&\le&C(n)\,\diam(\Om)\,|\Om|^2\,\de(\Om)^\a\,,
\end{eqnarray*}
so that \eqref{Omega riscalato H0 uguale n} gives us
\begin{eqnarray*}
|P(\Om)-P(G^*)|&=&|(n+1)|\Om|-P(G^*)|\le C(n)\,\big||\Om|-|G^*|\big|+C(n)\,\diam(\Om)\,|\Om|^2\,\de(\Om)^\a
\\
&\le&C(n)\,\diam(\Om)\,|\Om|^2\,\de(\Om)^\a\,,
\end{eqnarray*}
which proves \eqref{main thm perimetri stima proof} as $|\Om|\le C(n)\,P(\Om)$, and since, by an entirely similar argument,
\[
|P(G^*)-\#J\,P(B)|\le C(n)\,P(\Om)\,\diam(\Om)\,|\Om|\,\de(\Om)^\a\,.
\]
We conclude this step by proving \eqref{starstar}: indeed, by \eqref{Omega volume minore L+1}, \eqref{Omega diametro minore L} and \eqref{main thm sj-1}
\[
(L+1-a)|B|\ge|\Om|\ge |G^*|\ge\big(1-C(n,L)\,\de(\Om)^\a\big)|B|\#J\,,
\]
and thus we conclude by $\de(\Om)\le c(n,L,a)$.

\medskip

\noindent {\it Step four}: We prove that
\begin{eqnarray}
  \label{main thm onesided hd G*}
  \frac{\max_{x\in \pa G^*}\dist(x,\pa\Om)}{\diam(\Om)}\,\le C(n)\,\de(\Om)^\a\,.
\end{eqnarray}
We first notice that if $x_0\in\pa A_i$, then by \eqref{feps gi gradienti stima C0 in Ai}, $A_i\subset B_{x_i, r_2^i}$ and $r_2^i\le 2$, one has
\[
\Big|\nabla f_\e(x_0)-\frac{(x_0-x_i)}{n+1}\Big|\le C\,\de(\Om)^\a\,,
\]
so that, by $|x_0-x_i|\ge r_1^i\ge 1/2$, one finds
\[
|\nabla f_\e(x_0)|\ge c_1(n)\,,\qquad\forall x_0\in\pa A_i\,.
\]
Let $A_i^*$ be the set of points $x\in \Om_\e\setminus\ov{A_i}$ such that if $x_0\in\pa A_i$ denotes the projection of $x$ onto the convex set $A_i$, then the open segment joining $x$ to $x_0$ is entirely contained in $\Om_\e$ (and thus in $A_i^*$: in particular, $A_i^*$ is connected). By \eqref{D2feps positivo}, $f_\e(x)\le 0$ and $\pa A_i\subset\{f_\e=-3\rho\}$, we get
\begin{eqnarray*}
3\rho&\ge& f_\e(x)-f_\e(x_0)=\nabla f_\e(x_0)\cdot(x-x_0)+\frac12\int_0^1\,\nabla^2f_\e(tx+(1-t)x_0)[x-x_0,x-x_0]\,dt
\\
&\ge& |\nabla f_\e(x_0)| |x-x_0|-C_2(n)\,|x-x_0|^2\ge \frac{c_1(n)}2\,|x-x_0|\,,
\end{eqnarray*}
where we have used the fact that both $\nabla f_\e(x_0)$ and $x-x_0$ are orthogonal to $\pa A_i$ at $x_0$, and we have assumed that $|x-x_0|\le c_1(n)/2C_2(n)$. If we denote by
\[
I_d(X)=\big\{z\in\R^{n+1}:\dist(z,X)<d\big\}\,,\qquad X\subset \R^{n+1}\,,d>0\,,
\]
the $d$-neighborhood of a set $X$, then, setting $k_0(n)=c_1(n)/2C_2(n)$, we obtain
\[
I_{k_0(n)}(A_i)\cap A_i^*\subset I_{6\rho/c_1(n)}(A_i)\,.
\]
By connectedness of $A_i^*$ and by $\de(\Om)\le c(n,L)$, this proves that
\[
A_i^*\subset I_{6\rho/c_1(n)}(A_i)\,.
\]
Since $A_i\cc\Om_\e$ (thanks to \eqref{inclusioni livelli}), for every $x\in\pa A_i$, there exists $y\in \pa\Om_\e$ such the open segment joining $x$ and $y$ is entirely contained in $A_i^*$, and the length of this segment is bounded by $6\rho/c_1(n)$, so that
\[
\pa A_i\subset I_{6\rho/c_1(n)}(\pa\Om_\e)\subset I_{\e+(6\rho/c_1(n))}(\pa\Om)\,.
\]

\medskip

\noindent {\it Step five}: We construct a family of disjoint balls $\{B_{z_j,1}\}_{j\in J}$ such that if we set
\[
G=\bigcup_{j\in J}B_{z_j,1}\,,
\]
then
\begin{eqnarray}
  \label{step five 0}
  \frac{|\Om\Delta G|}{|\Om|}\le C(n)\,\diam(\Om)\,|\Om|\,\de(\Om)^\a\,,\qquad
  \frac{\max_{x\in \pa G}\dist(x,\pa\Om)}{\diam(\Om)}\le  C(n)\,|\Om|\,\de(\Om)^\a\,.
\end{eqnarray}
(Note that \eqref{step five 0} imply \eqref{main thm Omega meno G} and \eqref{main thm onesided hd} thanks to \eqref{Omega volume minore L+1} and \eqref{Omega diametro minore L}.) Indeed if we set $s_j'=\min\{s_j,1\}$, then $\{B_{x_j,s_j'}\}_{j\in J}$ is a family of disjoint balls such that $G'=\bigcup_{j\in J}B_{x_j,s_j'}$ satisfies $G'\subset\Om$ and
  \begin{eqnarray}\label{step five 1}
  \frac{|\Om\setminus G'|}{|\Om|}\le C(n)\,\diam(\Om)\,|\Om|\,\de(\Om)^\a\,,\qquad \frac{\max_{x\in \pa G'}\dist(x,\pa\Om)}{\diam(\Om)}\le  C(n)\,|\Om|\,\de(\Om)^\a\,,
\end{eqnarray}
thanks to \eqref{main thm Omega meno G^*}, \eqref{main thm onesided hd G*} and \eqref{main thm sj-1}. Next, let us fix $j_0\in J$ such that $1>s_{j_0}$. By translating each $x_j$ with $j\ne j_0$ into
\[
\hat x_j=x_j+(1-s_{j_0})\,\frac{x_j-x_{j_0}}{|x_j-x_{j_0}|}\,,
\]
and setting $\hat s_j=s_j$ if $j\ne j_0$, $\hat s_{j_0}=1$, $\hat x_{j_0}=x_{j_0}$, we find that $\{B_{\hat x_j,\hat s_j}\}_{j\in J}$ is a family of disjoint balls such that $\hat G=\bigcup_{j\in J}B_{\hat x_j,\hat s_j}$ satisfies
 \begin{eqnarray}\label{step five 2}
  \frac{|\Om\Delta \hat{G}|}{|\Om|}\le C(n)\,\diam(\Om)\,|\Om|\,\de(\Om)^\a\,,\qquad \frac{\max_{x\in \pa \hat{G}}\dist(x,\pa\Om)}{\diam(\Om)}\le  C(n)\,|\Om|\,\de(\Om)^\a\,,
\end{eqnarray}
thanks to \eqref{step five 1} and \eqref{main thm sj-1}. By iteratively repeating this procedure on each $j\in J$ such that $\hat{s}_j<1$ we finally construct a family $G$ with the required properties.

\medskip

\noindent {\it Step six}: In this step we complete the proof of Theorem \ref{thm main 1} up to statements (i) and (ii). To this end, we want to apply Lemma \ref{lemma allard} to $\Om$ and $G$. We first notice that for every $x\in\pa G$, thanks to \eqref{main thm onesided hd}, there exists $g(x)\in\pa\Om$ such that
\begin{equation}
\label{def of gx}
|x-g(x)|\le C(n)\,\diam(\Om)\,|\Om|\,\de(\Om)^\a\,.
\end{equation}
(The point $g(x)$ will play the role of $y$ in Lemma \ref{lemma allard}.) Setting $S_j=\pa B_{z_j,1}$ for $j\in J$, we define $\{r_x\}_{x\in\pa G}$ by the rule
\begin{equation}
  \label{def rx}
  r_x=\sup\Big\{r\in(0,2\de(\Om)^\b):B_{x,r}\cap (\pa G\setminus S_j)=\emptyset\Big\}\,,\qquad\mbox{if $j\in J$, $x\in S_j$}\,,
\end{equation}
and then set
\begin{equation}
  \label{rx sigma}
  \S^*=\big\{x\in \pa G:r_x\ge\de(\Om)^{\beta}\big\}\,,
\end{equation}
for some $\beta=\beta(n)\in(0,\a)$ to be suitable chosen later on, see \eqref{beta eccolo grande}. With $\S_\l$ defined as in \eqref{definizione Sigma lambda}, see the statement of Lemma \ref{lemma allard}, it is clear that we can choose $c_3(n)>0$ in such a way that
\begin{equation}
  \label{lambda nel quantitativo}
  \S_\l\subset\S^*\,,\qquad\mbox{for}\qquad\l=c_3(n)\,\de(\Om)^{\beta/2}\,.
\end{equation}
In particular, by Remark \ref{remark geometry of Sigma lambda} and by \eqref{starstar},
\begin{eqnarray}\label{grandi verita}
\begin{split}
&\mbox{$\pa G\setminus\S$ consists of at most $C(n)\#J$-many spherical caps}
\\
&\mbox{whose diameters are bounded by $C(n)\,\l\le C(n)\,\de(\Om)^{\beta/2}$}\,.
\end{split}
\end{eqnarray}
We now {\it claim} that for every $x\in\S_\l$, $\l$ as in \eqref{lambda nel quantitativo}, one can find $\rho_x$ such that,
\begin{gather}
\label{lallard taglia rhox proof}
\frac{c_0(n)\,\l^2}2\le\rho_x\le c_0(n)\,\l^2\,,
\\
\label{allard hp check}
    \s(\pa\Om,g(x),\rho_x)\le \s_0(n)\,\l(n)\,,
    \\
  \label{lallard taglia x-y proof}
    |x-g(x)|\le\frac{\s_0(n)\rho_x^2}2\,,
    \\
    \label{misteri della fede proof}
    |\Om\Delta G|\le C(n)\,\rho_x^{n+1}\,\sigma(\pa\Om,g(x),\rho_x)^{1/4n}\,,
\end{gather}
where $\s_0(n)$ and $\l(n)$ are as in Lemma \ref{lemma allard}. In proving the claim, the harder task is accommodating \eqref{allard hp check}, because it requires to control the perimeter convergence of $\Om$ to $G$ localized in balls in terms of the Alexandrov's deficit.

We now prove the claim. First of all we notice that in order to entail \eqref{lallard taglia rhox proof}, and thanks to $\S_\l\subset\S^*$, \eqref{def rx} and \eqref{rx sigma}, it is enough to pick $\rho_x$ satisfying
\begin{equation}
  \label{where to pick rhox}
  \frac{c_4(n)}2\,r_x\le \rho_x\le c_4(n)\,r_x\,,
\end{equation}
for a suitable constant $c_4\in(0,1)$. Next, we notice that by $\de(\Om)\le c(n,L)$ we can entail
\begin{equation}
  \label{rx rstarn}\sup_{x\in\S_\l}r_x\le 2\,\de(\Om)^\b\le r_*(n)\,,
\end{equation}
for an arbitrarily small constant $r_*(n)$. Provided $r_*(n)$ is small enough, then \eqref{rx rstarn}, \eqref{rn definizione x} and \eqref{rn uniform x} give us
\begin{eqnarray}
\label{perimetro G pallina}
  P(G;B_{x,r})\le
  (1+C(n)\,r^2)\,\om_n\,r^n\,,\qquad \forall r<r_x\,,
  \\
  \label{uniform G pallina}
  \sup\big\{|(p-x)\cdot\nu_{G}(x)|:p\in B_{x,r}\cap\pa G\big\}\le C(n)\,r^2\,,\qquad\forall r<r_x\,.
\end{eqnarray}
(We are going to use this bounds to quantify the size of $P(G;B_{g(x),\rho_x})$, see \eqref{spring} below.) By Chebyshev inequality and by \eqref{main thm Omega meno G^*}, we can pick $\rho_x$ satisfying \eqref{where to pick rhox} and
\begin{eqnarray}
\H^n\Big((\Om\setminus G^*)\cap \pa B_{g(x),\rho_x}\Big)&\le& C(n)\,|\Om|^2\,\diam(\Om)\,\de(\Om)^{\a-\b}\,,
\label{rhox 1}
\\\label{rhox 2}
\H^n\big((\pa\Om\cup\pa G^*\cup\pa G)\cap\pa B_{g(x),\rho_x}\big)&=&0\,.
\end{eqnarray}
(Notice that we are using $G^*$ in place of $G$ here, because $G^*$ is contained in $\Om$, and this will simplify a key computation based on the divergence theorem.) We include a brief justification of \eqref{rhox 1} for the sake of clarity: let us set
\[
W=\big\{\rho\in\big(\frac{c_4\,r_x}2,c_4\,r_x\big):\H^n\big((\Om\setminus G^*)\cap \pa B_{g(x),\rho}\big)\ge K(n)\,|\Om|^2\,\diam(\Om)\,\de(\Om)^{\a-\b}\big\}\,,
\]
then, with $C(n)$ as in \eqref{main thm Omega meno G^*} and for a suitably large value of $K(n)$, we have $\H^1(W)\le (C(n)/K(n))\,\de(\Om)^\beta<c_4\,\de(\Om)^\b/2\le c_4\,r_x/2$. Now let us consider the open sets
\[
U_j=\Big\{y\in\R^{n+1}:|y-z_j|<|y-z_{j'}|\qquad\forall j'\ne j\Big\}\,,\qquad j\in J\,,
\]
so that $B_{z_j,1}\subset U_j$ for every $j\in J$, and $\{U_j\}_{j\in J}$ is a partition of $\R^{n+1}$ modulo a $\H^n$-dimensional set. The boundary of each $U_j$ is contained into finitely many hyperplanes $\{L_{j,i}\}_{i=1}^{m_j}$, where $m_j\le \#\,J\le C(n)\,|\Om|$. Thus
\begin{equation}
  \label{stima j i}
  \#\{(j,i):j\in J\,,1\le i\le m_j\}\le C(n)\,|\Om|^2\,.
\end{equation}
We claim the existence of $v\in S^n$ and $t^*\in\R$ such that, setting $L_{j,i}^*=t^*\,v+L_{j,i}$,
\begin{eqnarray}
  \label{Ujstar 1x}
\H^n(L_{j,i}^*\cap(\Om\setminus G^*))&\le&C(n)\,|\Om|^6\,\diam(\Om)\,\de(\Om)^{\a/2}\,,
\\
  \label{Ujstar 2x}
|t^*|&\le&\de(\Om)^{\a/2}\,,
\\
  \label{Ujstar 1.5x}
\H^n\big(L_{j,i}^*\cap(\pa\Om\cup\pa G^*)\big)&=&0\,.
\end{eqnarray}
To choose $v$, we let $\nu_{j,i}$ be a normal vector to $L_{j,i}$, and require $v\in S^n$ to be such that
\begin{equation}
  \label{condizione v}
  |v\cdot\nu_{j,i}|\ge \frac{c_2(n)}{|\Om|^2}\,,\qquad\forall j\in J\,,1\le i\le m_j\,.
\end{equation}
(The existence of such $v$ is deduced by observing that if $\theta>0$, then each spherical stripe $Y_{j,i}^\theta=\{u\in S^n:|u\cdot\nu_{j,i}|<\theta\}$ satisfies $\H^n(Y_{j,i}^\theta)\le C(n)\theta$ so that by \eqref{stima j i}
\[
\H^n\Big(S^n\setminus\bigcup_{j\in J}\bigcup_{i=1}^{m_j}Y_{j,i}^\theta\Big)\ge\H^n(S^n)-C(n)\,|\Om|^2\,\theta>0\,,
\]
provided $\theta=c(n)/|\Om|^2$ for a suitably small value of $c(n)$.) We now find $t^*$. For a constant $M(n)$ to be properly chosen, let us set
\[
I_{j,i}=\Big\{t\in \R:|t|<\de(\Om)^{\a/2}\,,\H^n\big((\Om\setminus G^*)\cap(t\,v+L_{j,i})\big)\ge M(n)\,|\Om|^6\diam(\Om)\,\de(\Om)^{\a/2}\Big\}\,.
\]
If $\ell_{j,i}(y)=y\cdot\nu_{j,i}$ then $L_{j,i}=\{\ell_{j,i}=\beta_{j,i}\}$ for some $\beta_{j,i}$, while $t\,v+L_{j,i}=\{\ell_{j,i}=\beta_{j,i}+t\,v\cdot \nu_{j,i}\}$. By \eqref{main thm Omega meno G^*} and Fubini's theorem we find
\begin{eqnarray*}
  C_1(n)\,|\Om|^2\,\diam(\Om)\,\de(\Om)^\a&\ge&|\Om\setminus G^*|=\int_\R\H^n((\Om\setminus G^*)\cap\{\ell_{j,i}=s\})\,ds
  \\
  &\ge&|v\cdot\nu_{j,i}|\,\int_\R\H^n((\Om\setminus G^*)\cap\{\ell_{j,i}=\beta_{j,i}+t\,v\cdot \nu_{j,i}\})\,dt
  \\
  &\ge&|v\cdot\nu_{j,i}|\,\H^1(I_{j,i})\,M(n)\,|\Om|^6\diam(\Om)\,\de(\Om)^{\a/2}\,,
\end{eqnarray*}
that is, by \eqref{condizione v},
\[
|\Om|^2\,\H^1(I_{j,i})\le \frac{C_1(n)\,\de(\Om)^{\a/2}}{c_2(n)\,M(n)}\,,\qquad\forall j\in J\,,1\le i\le m_j\,.
\]
By combining this estimate with \eqref{stima j i}, we see that if $M(n)$ is large enough, then
\[
\H^1\Big((-\de(\Om)^{\a/2},\de(\Om)^{\a/2})\setminus\bigcup_{j\in J}\bigcup_{i=1}^{m_j}\cap I_{j,i}\Big)>0\,,
\]
that is, there exists $t^*$ such that \eqref{Ujstar 1x}, \eqref{Ujstar 2x} and \eqref{Ujstar 1.5x} hold. If we set $U_j^*=t^*\,v+U_j$, then $\{U_j^*\}_{j\in J}$ is a partition of $\R^{n+1}$ modulo a $\H^n$-dimensional set such that $\pa U_j^*$ is contained into the hyperplanes $\{L_{j,i}^*\}_{i=1}^{m_j}$ and such that
\begin{eqnarray}
  \label{Ujstar 1}
  \H^n\big(\pa U_j^*\cap(\Om\setminus G^*)\big)&\le& C(n)\,|\Om|^6\,\diam(\Om)\,\de(\Om)^{\a/2}\,,
  \\\label{Ujstar 1.5}
  \H^n\big(\pa U_j^*\cap(\pa\Om\cup\pa G^*)\big)&=&0\,,
  \\
  \label{Ujstar 2}
  \H^n\big((U_j^*\cap\pa G^*)\Delta S_j\big))&\le& C(n)\,\de(\Om)^{\a/2}\,.
\end{eqnarray}
Here, \eqref{Ujstar 1} and \eqref{Ujstar 1.5} are immediate from \eqref{Ujstar 1x} and \eqref{Ujstar 1.5x}. To prove \eqref{Ujstar 2}, let us recall that $S_j=\pa B_{x_j,s_j}\subset\ov{U_j}$, so that by translating the boundary hyperplanes of $U_j$ by $t^*\,v$ with $|t^*|\le \de(\Om)^{\a/2}$ we have possibly cut out from $S_j$ at most $m_j$-many spherical caps of $\H^n$-measure bounded above by
\[
C(n)\,|t^*|^{n/2}\le C(n)\,\de(\Om)^{n\,\a/4}\,,
\]
that is, thanks also to $m_j\le L$,
\[
\H^n((U_j^*\cap S_j)\Delta S_j)\le C(n)\,L\,\de(\Om)^{n\,\a/4}\le C(n)\,\de(\Om)^{\a/2}\,,
\]
thanks to $\de(\Om)\le c(n,L)$. By a similar argument, since $\H^n(S_{j'}\cap U_j)=0$ for $j\ne j'$, we have that $\H^n(U_j^*\cap S_{j'})\le C(n)\,\de(\Om)^{\a/2}$, and thus \eqref{Ujstar 2} is proved.

We now apply the divergence theorem to the vector field $y\mapsto(y-z_j)/|y-z_j|$ on the set of finite perimeter $(\Om\setminus G^*)\cap (U_j^*\setminus B_{g(x),\rho_x})$. Since $\Div((y-z_j)/|y-z_j|)=n/|y-z_j|$ for $y\ne z_j$ and $G^*\subset\Om$, by exploiting \cite[Theorem 16.3]{maggiBOOK}, one finds
\begin{eqnarray*}
  0&<&\int_{(U_j^*\setminus B_{g(x),\rho_x})\cap \pa\Om}\frac{y-z_j}{|y-z_j|}\cdot\nu_\Om(y)\,d\H^n_y
  -\int_{(U_j^*\setminus B_{g(x),\rho_x})\cap\pa G^*}\frac{y-z_j}{|y-z_j|}\cdot\nu_{G^*}(y)\,d\H^n_y
  \\
  &&+\int_{\pa(U_j^*\setminus B_{g(x),\rho_x})\cap(\Om\setminus G^*)}\frac{y-z_j}{|y-z_j|}\cdot\nu_{U_j^*\setminus B_{g(x),\rho_x}}(y)\,d\H^n_y
  \\
  &\le&P(\Om; U_j^*\setminus B_{g(x),\rho_x})+\H^n\Big((\pa U_j^*\cup\pa B_{g(x),\rho_x})\cap(\Om\setminus G^*)\Big)
  \\
  &&-P(G^*;U_j^*\setminus B_{g(x),\rho_x})+C(n)\,\de(\Om)^{\a/2}\,
\end{eqnarray*}
where in the last inequality we have used $\nu_{G^*}(y)\cdot[(y-z_j)/(y-z_j)]=1$ if $y\in S_j$ and \eqref{Ujstar 2}. By combining this last inequality with \eqref{rhox 1} and \eqref{Ujstar 1} we thus find
\begin{eqnarray*}
P(G^*;U_j^*\setminus B_{g(x),\rho_x})&\le&P(\Om;U_j^*\setminus B_{g(x),\rho_x})
\\
&&+C(n)\,|\Om|^2\,\diam(\Om)\,\Big(\de(\Om)^{\a-\b}+|\Om|^4\,\de(\Om)^{\a/2}\Big)\,.
\end{eqnarray*}
By adding up over $j\in J$, and since $\#\,J\le C(n)\,|\Om|$, we thus find
\begin{eqnarray*}
  P(G^*;\R^{n+1}\setminus B_{g(x),\rho_x})&\le&P(\Om; \R^{n+1}\setminus B_{g(x),\rho_x})\\
  &&+C(n)\,|\Om|^3\,\diam(\Om)\,\Big(\de(\Om)^{\a-\b}+|\Om|^4\,\de(\Om)^{\a/2}\Big)\,,
\end{eqnarray*}
which gives us, keeping in mind the construction used in step five to define $G$ starting from $G^*$, and also thanks to \eqref{main thm sj-1} and $\de(\Om)\le c(n)$,
\begin{eqnarray}
  \label{task}
  P(G;\R^{n+1}\setminus B_{g(x),\rho_x})&\le&P(\Om; \R^{n+1}\setminus B_{g(x),\rho_x})
  \\\nonumber
  &&+
C(n)\,|\Om|^3\,\diam(\Om)\,\Big(\de(\Om)^{\a-\b}+|\Om|^4\,\de(\Om)^{\a/2}\Big)\,,
\end{eqnarray}
which, combined with \eqref{main thm perimetri stima} and \eqref{rhox 2}, gives us
\begin{eqnarray}\label{chain}
P(\Om;B_{g(x),\rho_x})&=&P(\Om)-P(\Om;\R^{n+1}\setminus B_{g(x),\rho_x})
\\\nonumber
&\le&P(G)-P(G;\R^{n+1}\setminus B_{g(x),\rho_x})
+C(n)\,|\Om|^3\,\diam(\Om)\,\Big(\de(\Om)^{\a-\b}+|\Om|^4\,\de(\Om)^{\a/2}\Big)
\\\nonumber
&=&P(G;B_{g(x),\rho_x})+C(n)\,|\Om|^3\,\diam(\Om)\,\Big(\de(\Om)^{\a-\b}+|\Om|^4\,\de(\Om)^{\a/2}\Big)\,.
\end{eqnarray}
By \eqref{def of gx}, \eqref{Omega volume minore L+1}, \eqref{Omega diametro minore L}, and thanks to $\de(\Om)\le c(n,L)$, we entail $B_{g(x),\rho_x}\subset B_{x,r_x}$, so that, by definition of $r_x$,
\begin{eqnarray*}
P(G;B_{g(x),\rho_x})&=&\H^n(B_{g(x),\rho_x}\cap\pa G)=\H^n(B_{g(x),\rho_x}\cap S_j)
\\
&\le&(1+C(n)\,\rho_x)\,\H^n(B_{x,\rho_x}\cap S_j)=(1+C(n)\,\rho_x)\,P(G;B_{x,\rho_x})\,.
\end{eqnarray*}
By combining this inequality with \eqref{chain}, \eqref{perimetro G pallina}, \eqref{rx sigma} and \eqref{where to pick rhox} we find
\begin{eqnarray}\label{spring}
\frac{P(\Om;B_{g(x),\rho_x})}{\om_n\,\rho_x^n}-1&\le&
C(n)\,\rho_x+\frac{C(n)}{\rho_x^n}\,|\Om|^3\,\diam(\Om)\,\Big(\de(\Om)^{\a-\b}+|\Om|^4\,\de(\Om)^{\a/2}\Big)
\\
\nonumber
&\le&
C(n)\,\Big(\de(\Om)^\b+|\Om|^3\diam(\Om)\big(\de(\Om)^{\a-(n+1)\b}+|\Om|^4\de(\Om)^{(\a/2)-n\b}\big)\Big)\,.
\end{eqnarray}
By combining \eqref{spring} with \eqref{allard deficit} (the definition of $\s(\pa\Om,g(x),\rho_x)$), \eqref{Omega H fra nmezzi e 2n} and $\rho_x\le 2\de(\Om)^\b$, we find
\[
\s(\pa\Om,g(x),\rho_x)\le C(n)\,\Big(\de(\Om)^\b+|\Om|^3\diam(\Om)\big(\de(\Om)^{\a-(n+1)\b}+|\Om|^4\de(\Om)^{(\a/2)-n\b}\big)\Big)\,.
\]
For this estimate to be nontrivial we definitely need
\[
\a>\max\{(n+1)\b,2n\b\}=2n\b\,.
\]
Under this assumption we have $\a-(n+1)\b>(\a/2)-n\b$, and thus
\begin{eqnarray}\nonumber
\s(\pa\Om,g(x),\rho_x)&\le& C(n)\,\Big(\de(\Om)^\b+\,|\Om|^7\,\diam(\Om) \de(\Om)^{(\a/2)-n\b}\Big)
\\\label{cinotti}
&\le& C(n)\,|\Om|^7\,\diam(\Om)\,\de(\Om)^{\a/2(n+1)}\,,
\end{eqnarray}
where we have set
\begin{equation}
  \label{beta eccolo grande}
  \b=\frac{\a}{2(n+1)}\,,
\end{equation}
in order to have $\b=(\a/2)-n\b$. By $\de(\Om)\le c(n,L)$ and by \eqref{cinotti}, we have thus proved so far that for every $x\in\S_\l$ one can find $g(x)\in\pa\Om$ and $\rho_x\in(0,1)$ such that \eqref{lallard taglia rhox proof} and \eqref{allard hp check} hold.

In order to prove our claim, we are left to prove \eqref{lallard taglia x-y proof} and \eqref{misteri della fede proof}. By \eqref{def of gx}, \eqref{Omega volume minore L+1} and \eqref{Omega diametro minore L} we have $|x-g(x)|\le C(n,L)\,\de(\Om)^\a$ while $\s_0(n)\rho_x^2/2\ge c(n)\de(\Om)^{2\b}$ by \eqref{lallard taglia rhox proof}, so that \eqref{lallard taglia x-y proof} follows by $\de(\Om)\le c(n,L)$ thanks to the fact that $\a>2\b$. Similarly, concerning \eqref{misteri della fede proof} we see by \eqref{Omega H fra nmezzi e 2n} that $\s(\pa\Om,g(x),\rho_x)\ge n\,\rho_x$ so that
\[
\rho_x^{n+1}\,\sigma(\pa\Om,g(x),\rho_x)^{1/4n}\ge c(n)\,\rho_x^{n+2}=c(n)\,\de(\Om)^{\b(n+2)}\,,
\]
while \eqref{main thm Omega meno G} gives us $|\Om\Delta G|\le C(n,L)\,\de(\Om)^\a$, so that $\de(\Om)\le c(n,L)$ implies \eqref{misteri della fede proof} thanks to the fact that $\a>\b(n+2)$.

We thus proved our claim: for every $x\in\S_\l$ there exists $g(x)\in\pa\Om$ and $\rho_x$ satisfying \eqref{lallard taglia rhox proof}--\eqref{misteri della fede proof}. By \eqref{lambda nel quantitativo} and  $\de(\Om)\le c(n,L)$ we have that $\l\le\l(n)$, and thus we can apply Lemma \ref{lemma allard} to find a function $\psi:\S_\l\to\R$ such that,
  \begin{gather}\label{psilambda nomra c1gamma proof}
  \|\psi\|_{C^{1,\g}(\S_\l)}\le C(n,\g)\,,\qquad\forall \g\in(0,1)\,,
  \\\label{psilambda C0C1 proof}
  \l^{-2}\|\psi\|_{C^0(\S_\l)}+\|\nabla\psi\|_{C^0(\S_\l)}\le C(n)\,\max_{x\in\S_\l}\s(\pa\Om,g(x),\rho_x)^{1/4n}\,,
  \\\label{psilambda Nthetaintorno proof}
  (\Id+\psi\nu_G)(\S_\l)\subset\pa\Om\,.
  \end{gather}
(Notice that \eqref{psilambda nomra c1gamma proof} is \eqref{psilambda nomra c1gamma}.) Moreover, let us recall from the proof of Lemma \ref{lemma allard} (see \eqref{epici 3 x}) that if $x\in\S_\l\cap S_j$, then the function $\psi$ is actually defined on $\C_{x,\s_0\,\rho_x/4}^{\mu_x}\cap\pa G=\C_{x,\s_0\,\rho_x/4}^{\mu_x}\cap S_j$, with
\begin{equation}\label{bellazi}
  \begin{split}
    \C_{x,\s_0\,\rho_x/8}^{\mu_x}\cap\pa\Om\subset(\Id+\psi\,\nu_G)(\C_{x,\s_0\,\rho_x/4}^{\mu_x}\cap\pa G)
\subset \C_{x,\s_0\,\rho_x/2}^{\mu_x}\cap\pa\Om\,\,,
\\
\|\psi\|_{C^1(\C_{x,\s_0\,\rho_x/4}^{\mu_x}\cap\pa G)}\le C(n)\s(\pa\Om,g(x),\rho_x)^{1/4n}\,.
  \end{split}
\end{equation}
Now, by \eqref{def of gx}, $\de(\Om)\le c(n,L)$, and $\s_0\,\rho_x/8\ge c(n)\,\de(\Om)^\b$ we find
\[
g(x)\in \C_{x,\s_0\,\rho_x/8}^{\mu_x}\cap\pa\Om\,,
\]
and thus, by the first inclusion in \eqref{bellazi}, there exists $y\in \C_{x,\s_0\,\rho_x/4}^{\mu_x}\cap\pa G$ such that
\[
g(x)=y+\psi(y)\nu_G(y)\,.
\]
By \eqref{bellazi}, \eqref{cinotti}, and $\de(\Om)\le c(n,L)$ we can definitely ensure
\begin{equation}
  \label{bellazi2}
  \|\psi\|_{C^1(\C_{x,\s_0\,\rho_x/4}^{\mu_x}\cap\pa G)}\le c(n)\,,
\end{equation}
so that, taking into account that \eqref{uniform G pallina} gives $|(y-x)\cdot\nu_G(y)|\le C(n)\,|y-x|^2$, we find that
\[
|g(x)-x|^2\ge |x-y|^2+|\psi(y)|^2-2\|\psi\|_{C^0(\C_{x,\s_0\,\rho_x/4}^{\mu_x}\cap\pa G)}\,|\nu_G(y)\cdot(y-x)|\ge\frac{|x-y|^2}2+|\psi(y)|^2\,.
\]
Again by \eqref{def of gx} we conclude $|x-y|+|\psi(y)|\le C(n)\,|\Om|\,\diam(\Om)\,\de(\Om)^\a$, and thus, provided $c(n)\le 1$ in \eqref{bellazi2},
$|\psi(x)|\le C(n)\,|\Om|\,\diam(\Om)\,\de(\Om)^\a$. We have thus improved the $C^0$-bound on $\psi$ in \eqref{psilambda C0C1 proof}, by showing that
\begin{equation}
  \label{stima psi C0 ottimale}
  \|\psi\|_{C^0(\S_\l)}\le C_3(n)\,|\Om|\,\diam(\Om)\,\de(\Om)^\a\,.
\end{equation}
By combining \eqref{cinotti} with \eqref{psilambda C0C1 proof} and $\diam(\Om)\le C(n)\,P(\Om)\le C(n)|\Om|$, we obtain
\begin{equation}
  \label{stima psi C1 come e venuta}
  \|\nabla\psi\|_{C^0(\S_\l)}\le C(n)\,|\Om|^{2/n}\,\de(\Om)^{\b/4n}\,.
\end{equation}
(By \eqref{stima psi C0 ottimale} and \eqref{stima psi C1 come e venuta} we deduce \eqref{stima psi intro 2}.) Next, by \eqref{grandi verita} and by $\#\,J\le C(n)\,|\Om|\le C(n)\,P(\Om)$, we have
\begin{equation}
  \label{buchetti}
  \H^n(\pa G\setminus\S_\l)\le C(n)\,P(\Om)\,\de(\Om)^{n\,\b/2}\,,
\end{equation}
while by the area formula
\[
\H^n\big((\Id+\psi\,\nu_G)(\S_\l)\big)=\int_{\S_\l}\sqrt{(1+\psi)^{2n}+(1+\psi)^{2(n-1)}|\nabla\psi|^2}\,,
\]
so that
\begin{eqnarray*}
  \big|\H^n(\S_\l)-\H^n\big((\Id+\psi\,\nu_G)(\S_\l)\big)\big|&\le&
  P(G)\big(\|\psi\|_{C^0(\S_\l)}+\|\nabla\psi\|_{C^0(\S_\l)}^2\big)\\
  &\le&C(n)\,P(\Om)\,|\Om|^{4/n}\,\de(\Om)^{\b/2n}\,,
\end{eqnarray*}
where in the last inequality we have used $P(G)\le 2\,P(\Om)$ (which follows by $\de(\Om)\le c(n,L)$ and \eqref{main thm perimetri stima proof}) together with \eqref{stima psi C0 ottimale} and \eqref{stima psi C1 come e venuta}. By \eqref{main thm perimetri stima proof}, \eqref{buchetti}, and
\begin{eqnarray*}
    &&\H^n\big(\pa\Om\setminus(\Id+\psi\,\nu_G)(\S_\l)\big)
    \\
    &\le&P(\Om)-P(G)+\H^n(\pa G\setminus\S_\l)
    +\big|\H^n(\S_\l)-\H^n\big((\Id+\psi\,\nu_G)(\S_\l)\big)\big|
  \end{eqnarray*}
we thus obtain
  \begin{equation}
    \label{stima bordo omega meno immagine di sigma proof}
    \frac{\H^n\big(\pa\Om\setminus(\Id+\psi\,\nu_G)(\S_\l)\big)}{P(\Om)}\le C(n)\,|\Om|^{4/n}\,\de(\Om)^{\b/2n}\,.
  \end{equation}
  Now let $x_0\in\pa\Om$ and $\tau>0$ be such that
  \[
  \tau=\max_{x\in\pa\Om}\dist(x,\pa G)=\dist(x_0,\pa G)\,,
  \]
  and assume that $\tau\ge 2\,\theta$ for $\theta=C_3(n)\,\diam(\Om)\,|\Om|\,\de(\Om)^\a$. By \eqref{stima psi C0 ottimale} and since $B_{x_0,\tau}\subset \R^{n+1}\setminus \pa G$, one has
  \[
  B_{x_0,\tau-\theta}\subset \R^{n+1}\setminus I_\theta(\pa G)\subset \R^{n+1}\setminus (\Id+\psi\nu_G)(\S_\l)\,,
  \]
  so that, thanks to \eqref{Omega stima densita perimetro basso},
  \begin{eqnarray*}
  c(n)\,\min\{1,\tau-\theta\}^n\le P(\Om,B_{x_0,\tau-\theta})
  \le\H^n(\pa\Om\setminus (\Id+\psi\nu_G)(\S_\l))
  \le C(n)\,P(\Om)\,|\Om|^{4/n}\,\de(\Om)^{\b/2n}\,.
  \end{eqnarray*}
  By $\de(\Om)\le c(n,L)$ we thus find $\min\{1,\tau-\theta\}=\tau-\theta$, and thus
  \[
  \tau\le\theta+C(n)\,P(\Om)^{1/n}\,|\Om|^{4/n^2}\,\de(\Om)^{\b/2n^2}\le C(n)\,P(\Om)^{1/n}\,|\Om|^{4/n^2}\,\de(\Om)^{\b/2n^2}\,.
  \]
  Since $P(\Om)\le C(n)|\Om|\le C(n)\,\diam(\Om)^{n+1}$ and $\diam(\Om)\le C(n)|\Om|$, we find $P(\Om)^{1/n}\le C(n)\,\diam(\Om)\,|\Om|^{1/n}$ and thus conclude that
  \begin{equation}
    \label{step five 0 twin}
      \frac{\max_{x\in\pa\Om}\dist(x,\pa G)}{\diam(\Om)}\le C(n)\,|\Om|^{3/n}\,\de(\Om)^{\b/2n^2}\,.
  \end{equation}
  where have simplified some powers on $|\Om|$ by noticing that $(4/n^2)+(1/n)\le 3/n$. By combining \eqref{step five 0} and \eqref{step five 0 twin} we obtain \eqref{main thm hd stima}.

  \medskip

  \noindent {\it Step seven}: We complete the proof of the theorem. By contradiction, and by definition of $\S_\l$, if $\#\,J\ge 2$ and $\dist(S_j,S_\ell)\ge 4\l$ for every $\ell\ne j$, then $S_j\subset\S_\l$, and thus $\Gamma_j=(\Id+\psi\nu_G)(S_j)\subset\pa\Om$, with $\Gamma_j$ connected. Since $\pa\Om$ is connected, we conclude that $\pa\Om=\Gamma_j$, thus that $\#\,J=1$. This proves that if $\#\,J\ge 2$, then for every $j\in J$ there is $\ell\ne j$ such that $\dist(S_j,S_\ell)\le 4\l$, and then (i) follows by $\diam(\Om)\ge c(n)\,|\Om|^{1/(n+1)}\ge c(n)$. In order to prove (ii), let us assume that $\#\,J\ge 2$, and let $j\ne\ell\in J$ be such that \eqref{sjsjprimo tangenti} holds. Let us set
  \[
  x=\frac{z_j+z_\ell}2=\frac{\bar z_j+\bar z_\ell}2\,,
  \]
  where $\bar z_j$ and $\bar z_\ell$ belong to the closed segment joining $z_j$ and $z_\ell$ and are such that
  \[
  \pa B_{\bar z_j,1}\cap  \pa B_{\bar z_\ell,1}=\{x\}\,.
  \]
  In this way, for every $r\in(0,\k)$ and thanks to \eqref{sjsjprimo tangenti} we have
  \begin{eqnarray*}
    C(n)\,r^{n+2}\ge\big|B_{x,r}\setminus\big(B_{\bar z_j,1}\cup  B_{\bar z_\ell,1}\big)\big|
    \ge\big|B_{x,r}\setminus\big(B_{z_j,1}\cup  B_{z_\ell,1}\big)\big|-C(n)\,\de(\Om)^{\a/4(n+1)}\,.
  \end{eqnarray*}
  By \eqref{main thm Omega meno G} we have
  \begin{eqnarray*}
    \big|B_{x,r}\setminus\big(B_{z_j,1}\cup  B_{z_\ell,1}\big)\big|\ge|B_{x,r}\setminus G|\ge|B_{x,r}\setminus\Om|-C(n)\,L^3\,\de(\Om)^\a\,,
  \end{eqnarray*}
  so that by \eqref{stima densita fuori intro} we finally get
  \[
  \k\,|B|\,r^{n+1}\le C(n)\,\big(L^3\de(\Om)^\a+\de(\Om)^{\a/4(n+1)}+r^{n+2}\big)\,,\qquad\forall r<\k\,.
  \]
  We now use this inequality with $r=\vartheta\,\k$, with $\vartheta=\vartheta(n)\in(0,1)$ to be properly chosen. By $\de(\Om)\le c(n,L,\k)$ we can entail,
  \[
  L^3\de(\Om)^\a+\de(\Om)^{\a/4(n+1)}\le \vartheta^{n+2}\,\k^{n+2}\,,
  \]
  thus finding $\k^{n+2}\,|B|\,\vartheta^{n+1}\le C(n)\,\vartheta^{n+2}\,\k^{n+2}$. This is of course a contradiction if we pick $\vartheta$ small enough. The proof of Theorem \ref{thm main 1} is complete.
\end{proof}

\section{An application to capillarity-type energies}\label{section capillarity}

\begin{proof}[Proof of Proposition \ref{corollary main stationary}]
  It is well-known that \eqref{stationary set intro} implies the existence of a constant $\l\in\R$ such that
  \begin{equation}
  \label{stationary set}
  \int_{\pa\Om}\,\Div^{\pa\Om}\,X+\int_\Om\,g\,(X\cdot\nu_\Om)=\l\int_{\pa\Om} (X\cdot\nu_\Om)\,,\qquad\forall X\in C^\infty_c(\R^{n+1};\R^{n+1})\,.
  \end{equation}
Since $R_0>0$ is such that $\Om\cc B_{2R_0}$, by testing \eqref{stationary set} with $X(x)=\vphi(x)\,x$ for some $\vphi\in C^\infty_c(B_{2R_0})$ with $\vphi=1$ on $\ov{\Om}$, one easily obtains that
  \begin{equation}
  \label{valore di lambda}
  (n+1)\,|\Om|\,\l=n\,P(\Om)+\int_\Om\,\Div(x,g(x))\,dx\,.
  \end{equation}
  At the same time, since $\pa\Om$ is of class $C^2$, \eqref{stationary set} implies that $H+g=\l$ on $\pa\Om$, which combined with \eqref{valore di lambda} gives
  \[
  \frac{H(x)-H_0}{H_0}=\frac1{H_0}\Big(-g(x)+\frac{\int_\Om\,\Div(x\,g(x))\,dx}{(n+1)|\Om|}\Big)\qquad\Rightarrow\qquad\de(\Om)\le \frac{C\,\|g\|_{C^1(B_{R_0})}}{H_0}\,.
  \]
  By the isoperimetric inequality,
  \[
  H_0=\frac{n\, P(\Om)}{(n+1)\,|\Om|}\ge\frac{n\,(n+1)\,|B|^{1/(n+1)}\,|\Om|^{n/(n+1)}}{(n+1)\,|\Om|}=n\,\Big(\frac{|B|}{|\Om|}\Big)^{1/(n+1)}\,.
  \]
  so that $\de(\Om)\le C(n)\,\|g\|_{C^1(B_{R_0})}\,m^{1/(n+1)}$, that is \eqref{delta stazionario}.
\end{proof}

\appendix

\section{Almost-constant mean curvature implies almost-umbilicality}\label{appendix umbilical} The purpose of this appendix is to discuss the relation between the Alexandrov's deficit $\de(\Om)$ and the size of the traceless part $\mathring{A}$ of the second fundamental form $A$ of $\pa\Om$. Having in mind the quantitative results for almost-umbilical surfaces by De Lellis-M\"uller \cite{delellismuller1,delellismuller2} and Perez \cite{perez}, we seek for a control of $\mathring{A}$ in $L^p(\pa\Om)$ for some $p\ge n$. In this direction, we have the following proposition, where $\eta(\Om)$ denotes the Heintze-Karcher deficit of $\Om$, see \eqref{def eta}.

\begin{proposition}\label{proposition mr}
  If $\Om\subset\R^{n+1}$ ($n\ge 2$) is a bounded connected open set with $C^2$-boundary, $H>0$ on $\pa\Om$, and $\de(\Om)\le 1/2$, then
  \begin{equation}
    \label{montielros inq}
    C(n)\,\big(P(\Om)\,\eta(\Om)\big)^{1/(n+1)}\,\|A\|_{L^{p^\star}(\pa\Om)}
    \ge\|\mathring{A}\|_{L^p(\pa\Om)}\,,\qquad\forall p\in[1,n+1]\,,
  \end{equation}
  where $p^\star=(n+1)p/[(n+1)-p]$ if $p<n+1$, and $p^\star=+\infty$ otherwise.
\end{proposition}

Before proving Proposition \ref{proposition mr}, let us discuss \eqref{montielros inq} in connection with the above mentioned results for almost-umbilical surfaces. Let us consider
\[
\theta_p(\Om)=\inf_{\l\in\R}\|A-\l\,\Id\|_{L^p(\pa\Om)}\,,\qquad p\ge 1\,,
\]
as a measure of the non-umbilicality of $\pa\Om$. In \cite[Theorem 1.1]{perez} it is shown that if $\Om$ is a bounded connected open set with smooth boundary such that
\[
P(\Om)=P(B)\,,\qquad \|A\|_{L^p(\pa\Om)}\le K\,,
\]
for some $p\in(n,\infty)$, then
\begin{equation}
  \label{perez}
  \inf_{\l\in\R}\|A-\l\,\Id\|_{L^p(\pa\Om)}\le C(n,p,K)\,\|\mathring{A}\|_{L^p(\pa\Om)}\,.
\end{equation}
Moreover, in \cite[Corollary 1.2]{perez} it is shown that for every $\e>0$ there exists $c=c(n,p,K,\e)$ such that if
\begin{equation}
  \label{perez2}
  \|\mathring{A}\|_{L^p(\pa\Om)}\le c\qquad\Rightarrow\qquad \inf_{x\in\R^{n+1}}\hd(\pa\Om,\pa B_{x,1})<\e\,.
\end{equation}
By combining \eqref{montielros inq} and \eqref{perez} with $\eta(\Om)\le\de(\Om)$, we deduce that, if $P(\Om)=P(B)$, $H>0$ and $\de(\Om)\le 1/2$, then
\[
\theta_p(\Om)\le C(n,p,\|A\|_{L^{p^\star}(\pa\Om)})\,\de(\Om)^{1/(n+1)}\,.
\]
If, in addition, $\de(\Om)\le c(n,p,\|A\|_{L^{p^\star}(\pa\Om)},\e)$, then, by \eqref{perez2},
\[
\inf_{x\in\R^{n+1}}\hd(\pa\Om,\pa B_{x,1})<\e\,.
\]
A similar comparison is possible with the results of De Lellis and M\"uller, which pertain the case $n=p=2$. In conclusion, the use of almost-umbilicality in attacking Theorem \ref{thm main 1} does not seem to provide one with a starting point as effective as the one based on the study of the torsion potential discussed in section \ref{section proof of the main theorem}. We finally prove the above proposition.

\begin{proof}[Proof of Proposition \ref{proposition mr}]
  This is consequence of the proof of the Heintze-Karcher inequality by Montiel-Ros \cite{montielros}, which we now recall for the reader's convenience. For each $x\in\pa\Om$ let $\{\k_i(x)\}_{i=1}^n$ be the principal curvatures of $\pa\Om$ at $x$, so that, if we set $\k=\max_{1\le i\le n}\k_i$, then $\k\ge H/n>0$ on $\pa\Om$. Let us consider the set
\[
\Gamma=\Big\{(x,t)\in\pa\Om\times(0,\infty):t\le\frac1{\k(x)}\Big\}\,,
\]
and the function $g:\pa\Om\times(0,\infty)\to\R^{n+1}$
\[
g(x,t)=x-t\,\nu_\Om(x)\,,\qquad (x,t)\in\pa\Om\times(0,\infty)\,.
\]
We claim that $\Om\subset g(\Gamma)$. Indeed, given $y\in\Om$ let $x\in\pa\Om$ be such that $|x-y|=\dist(y,\pa\Om)$. If $\g$ is a curve in $\pa\Om$ with $\g(0)=x$ and $\g'(0)=\tau\in S^n$, then we obtain
\[
(\g-y)\cdot\g'=0\,,\qquad (\g(0)-y)\cdot\g''(0)+\g'(0)^2\ge 0\,.
\]
In particular, there exists $t>0$ such that $x=y-t\,\nu_\Om(y)$, and it must be $1-\k(y)\,t\ge 0$. We now combine $\Om\subset g(\Gamma)$ with the area formula, the arithmetic-geometric mean inequality, and $\k\ge H/n$, to prove the Heintze-Karcher inequality
\begin{eqnarray*}
  |\Om|&\le&|g(\Gamma)|\le\int_{g(\Gamma)}\,\H^0(g^{-1}(y))\,dy=\int_{\Gamma}\,J^{\Gamma}\,g(x,t)\,d\H^n(x)\,dt
  \\
  &=&\int_{\pa\Om}\,d\H^n\,\int_0^{1/\k}\,\prod_{i=1}^n(1-t\,\k_i)\,dt
  \\
  &\le&\int_{\pa\Om}\,d\H^n\,\int_0^{1/\k}\,\Big(\frac1n\sum_{i=1}^n(1-t\,\k_i)\Big)^n\,dt
  =\int_{\pa\Om}\,d\H^n\,\int_0^{1/\k}\,\Big(1-t\,\frac{H}n\Big)^n\,dt
  \\
  &\le&\int_{\pa\Om}\,d\H^n\,\int_0^{n/H}\,\Big(1-t\,\frac{H}n\Big)^n\,dt=\frac{1}{n+1}\int_{\pa\Om}\frac{n}{H}\,.
\end{eqnarray*}
This chain of inequalities implies of course the following identity
\begin{eqnarray}\label{controllo}
  \int_{\pa\Om}\frac{n}H\,d\H^n\,\frac{\eta(\Om)}{n+1}&=&\frac1{n+1}\,\int_{\pa\Om}\,\frac{n}{H}\,\Big(1-\frac1{\k}\,\frac{H}n\Big)^{n+1}\,
  +\int_{\Gamma}\,\mu_A^n-\mu_G^n\,
  \\\nonumber
  &&+\int_{g(\Gamma)}\,\big(\H^0(g^{-1}(y))-1\big)\,dy+\big(|g(\Gamma)|-|\Om|\big)\,,
\end{eqnarray}
where for every $(x,t)\in\Gamma$ we have set $\mu_i(x,t)=1-t\,\k_i(x)$ (note that $\mu_i\ge0$ on $\Gamma$) and
\[
\mu_A=\frac1{n}\sum_{i=1}^n\mu_i\,,\qquad \mu_G=\prod_{i=1}^n\mu_i^{1/n}\,.
\]
The first two terms on the right-hand side of \eqref{controllo} provide some control on $\mathring{A}$. Since $|\mathring{A}(x)|\le C(n)\,|\k(x)-(H(x)/n)|$ for every $x\in\pa\Om$, by looking at the first term, we get
\[
C(n)\,\eta(\Om)\,\int_{\pa\Om}\frac{n}H\ge\,\int_{\pa\Om}\frac{n}{H}\,\Big(\frac{|\mathring{A}|}{\k}\Big)^{n+1}\,,
\]
that is, by exploiting $\de(\Om)\le 1/2$ to infer $H_0/2\le H(x)\le 2H_0$ for every $x\in\pa\Om$,
\[
C(n)\,P(\Om)\,\eta(\Om)\ge\,\int_{\pa\Om}\Big(\frac{|\mathring{A}|}{\k}\Big)^{n+1}\,,
\]
We thus conclude the proof by H\"older inequality.
\end{proof}

\bibliography{references}
\bibliographystyle{is-alpha}
\end{document}